\documentclass[11pt]{article}
\usepackage[utf8]{inputenc}

\usepackage{amsthm, amsfonts, amssymb, amsmath,enumitem, natbib, bbm, mathabx}
\usepackage{cases, amsthm}
\usepackage{fullpage}
\usepackage{wasysym} 
\usepackage{fancyhdr}
\usepackage{graphicx}
\usepackage{bbm}
\usepackage{caption}
\usepackage{subcaption}
\usepackage{bm}
\usepackage{authblk}
 \usepackage{setspace}
\usepackage[colorlinks=true, pdfstartvie w=FitV,
linkcolor=blue, citecolor=blue, urlcolor=blue]{hyperref}
\usepackage{comment}
\usepackage{mathtools}
\usepackage{amssymb}

\newtheorem{definition}{Definition}
\newtheorem{theorem}{Theorem}

\newtheorem{lemma}{Lemma}

\newtheorem{assumption}{Assumption}
\newtheorem{result}{Result}

\newtheorem{remark}{Remark}

\usepackage[linesnumbered,ruled,vlined]{algorithm2e}

\usepackage{mathrsfs}

\usepackage{xcolor}
\usepackage{cleveref}
\usepackage[top=1in, bottom=1in, left=1in, right=1in]{geometry}
\title{Finite sample valid confidence sets of mode}

\author[1]{Manit Paul}
\author[2]{Arun Kumar Kuchibhotla}
\affil[1]{Department of Statistics \& Data Science, University of Pennsylvania}
\affil[2]{Department of Statistics \& Data Science, Carnegie Mellon University}
\date{}

\begin{document}

\maketitle

\begin{abstract}
Estimating the mode of a unimodal distribution is a classical problem in statistics. Although there are several approaches for point-estimation of mode in the
literature, very little has been explored about the interval-estimation of mode. Our work proposes a collection of novel methods of obtaining finite sample valid confidence set of the mode of a unimodal distribution. We analyze the behaviour of
the width of the proposed confidence sets under some regularity assumptions of
the density about the mode and show that the width of these confidence sets shrink to zero near optimally. Simply put, we show that it is possible to build finite sample valid confidence sets for the mode that shrink to a singleton as sample size increases. We support the theoretical results by showing the performance of the proposed methods on some synthetic data-sets. We believe that our confidence sets can be improved both in construction and in terms of rate.
\begin{comment}
    The construction of confidence intervals for the median based on independent and identically distributed (i.i.d.) observations is an age-old problem in statistics for which there are numerous solutions available in the literature. Arguably one reason for the popularity of this problem is that one can construct distribution-free finite sample valid confidence intervals for the median using order statistics, i.e., with no assumptions on the underlying distribution whatsoever. This, however, only applies to coverage validity. The other side of this problem is the width. To our knowledge, the width of the distribution-free confidence interval is studied only under standard assumptions that require the existence of a Lebesgue density that is bounded away from zero at the median. In this paper, we advance the literature by studying the width of the distribution-free confidence interval under non-standard assumptions. A striking conclusion is that the width of the confidence interval when scaled appropriately converges in distribution to a non-degenerate random variable, which is unlike regular intervals where the limit is degenerate. Finally, we consider the application of the results to the construction of confidence intervals for general parameters/functionals improving the existing procedure HulC. 
\end{comment}
\end{abstract}

\section{Introduction and Motivation}
\label{sec:introduction}
Mode estimation is a very well-studied problem in Statistics. Particularly for unimodal distributions, mode is a fundamental functional. A unimodal univariate distribution function $F$ is said to have a mode at $\theta_0$ if $F(\cdot)$ is a convex function in the interval $\left(-\infty, \theta_0 \right]$ and a concave function in the interval $\left[ \theta_0, \infty \right)$. This definition automatically implies that $F$ is absolutely continuous everywhere, except possibly a mass at $\theta_0$. In many results of this work we shall make the slightly stronger assumption that $F$ is an absolutely continuous distribution function. Unimodality of $F$ then implies the existence of a density $f$ which is non-decreasing on $(-\infty, \theta_0]$ and non-increasing on $(\theta_0, \infty]$. In other words, 
\begin{equation}
\label{eq:mode_def}
\theta_0 \in \mbox{arg max}_xf(x). 
\end{equation}
One of the earliest works of mode estimation is due to \cite{parzen1962estimation}. This paper adopted a two-stage process for estimating mode: obtain a kernel density estimate of the underlying density function and then obtain the plug-in estimate by finding the maximiser of the density estimate. This approach of estimating the mode using a non-parametric density estimate was studied in several subsequent works, the most prominent being \cite{chernoff1964estimation}, \cite{romano1988weak}. In the multivariate setting, prominent works that employ the kernel density approach are \cite{konakov1973asymptotic}, \cite{samanta1973nonparametric}, \cite{mokkadem2003law}, \cite{dasgupta2014optimal}. Along the similar lines, \cite{venter1967estimation} use the spacings to get a consistent estimate of mode: given data $X_1, \cdots X_n \stackrel{iid}{\sim} F$ where $F$ is a unimodal univariate distribution, the mode $\theta_0$ is estimated by the order statistic $\widehat \theta_{(K_n)}$ where $K_n = \mbox{arg min}_{j} \{X_{(j + r_n)} - X_{(j- r_n)}: r_n + 1 \leq j \leq n- r_n\}$. There is a large class of papers which use such spacings-based density estimate for obtaining an estimate of the mode, see \cite{dalenius1965mode}, \cite{grenander1965some}, \cite{sager1975consistency}, \cite{hall1982asymptotic}. 

A separate direction of work estimates the mode by using the fact that the mode of a unimodal distribution always lies in a high probability region (i.e. the smallest set $\mathcal{S}$ such that $\mathbb{P}(X \in \mathcal{S}) \geq \tau$ for some $0 < \tau < 1$). \cite{wegman1971note}, a notable work in this domain, estimates the mode by identifying a high probability region, which is found by determining the interval of length $a_n$ (some sequence) containing the highest number of observations. A modified version of this method is considered in \cite{robertson1974iterative}, where for a given function $k(n)$, the smallest interval contaning $k(n)$ many observations is selected. Thereafter the smallest sub-interval of the selected interval is chosen containing $k(k(n))$ many observations and this process is iterated multiple times. A generalization of this process for the multivariate setting is discussed in \cite{sager1979iterative}. A similar idea is used to construct an estimator of mode by \cite{bickel2002robust}, \cite{bickel2006fast} in the univariate case, and by \cite{arias2022estimation} in the multivariate case. For the interested readers, we refer to \cite{chacon2020modal} for a brief review of the methods of mode-estimation, and to \cite{dharmadhikari1988unimodality} for a treatise on unimodality. 

Despite extensive research on point estimation of the mode, interval estimation of mode has received much less attention. Some of the papers cited above derive the asymptotic distribution of the estimator of mode under some regularity assumptions on the behavior of the density about the mode, which can be used to obtain asymptotically valid confidence interval of the mode. However these asymptotically valid confidence intervals often have miserable performance for moderate sample sizes because of the slow rate of convergence of the estimator of mode. Thus developing finite sample valid confidence sets of the mode is of utmost importance. One of the most remarkable works on the interval estimation of mode is by \cite{edelman1990confidence}. If $X \sim F$, a unimodal univariate absolutely continuous distribution function with mode $\theta_0$, then for $0< \alpha < 1$ and any $a \in \mathbb{R}$, \cite{edelman1990confidence} states that, 
\begin{equation}
    \label{eq:edelman}
    \mathbb{P}\left(X - \left(\frac{2}{\alpha} - 1 \right)| X - a| \leq \theta_0 \leq X + \left(\frac{2}{\alpha} + 1 \right) | X - a|  \right) \geq 1 - \alpha. 
\end{equation}
This means than unlike other well known functionals like median, it is possible to obtain a valid confidence interval of the mode of a unimodal distribution using just a single observation. Under the same setting, \cite{andvel1991confidence} discusses an analogous method of constructing confidence interval of mode based on a single observation. Construction of confidence interval based on a single observation has also been studied in \cite{krauth1992confidence} for the case of discrete distributions. Apart from these works, \cite{lanke1974interval} has proposed the following valid confidence interval of mode based on $n$ iid observations $X_1, \cdots, X_n $ from a unimodal univariate distribution with mode $\theta_0$, 
\begin{equation}
    \label{eq:lanke}
    \mathbb{P}\left(X_{(1)} - \lambda_{n,\alpha}(X_{(n)} - X_{(1)}) \leq \theta_0 \leq X_{(n)} + \lambda_{n,\alpha}(X_{(n)} - X_{(1)}) \right) \geq 1- \alpha,
\end{equation}
for $\lambda_{n,\alpha} = \alpha^{-1/(n-1)} - 1$. This implies that if the underlying distribution has unbounded support then the width of the confidence interval proposed in \cite{lanke1974interval} diverges to infinite as the sample size increases. To our knowledge, there are no prior works that have proposed finite sample valid confidence intervals of the mode, with width shrinking to zero with increasing sample size. Our work advances this field by introducing methods for obtaining finite sample valid confidence sets of the mode of a unimodal distribution and proving that their widths diminish to zero at a ``near-optimal" rate.

\subsection{Contributions and Organization}
Given that finite-sample valid confidence intervals for the mode of a unimodal distribution can be constructed (\cite{lanke1974interval}, \cite{edelman1990confidence}), we do not explore asymptotically valid confidence interval methods like bootstrap or subsampling in this paper. The main contributions of this paper are as follows:
\begin{enumerate}
    \item Given a data-set of $n$ iid observations from a unimodal univariate distribution, we propose novel methods of constructing finite sample valid confidence set of the mode of the distribution using (i) spacings between the order statistics, and (ii) M-estimation procedure. Both of the proposed methods are adaptive to the behavior of the density about the mode i.e.\ under suitable assumptions (refer \cite{arias2022estimation}) these methods yield confidence sets whose width shrink to zero at the minimax optimal rate (upto a logarithmic factor) without any prior knowledge of the regularity of the density about the mode. 
    \item We extend the results of \cite{edelman1990confidence}, \cite{andvel1991confidence} on constructing a valid confidence interval of mode of a unimodal univariate distribution using a single observation with the aid of \eqref{eq:edelman} to provide an algorithm for constructing a finite sample valid confidence set of the mode using a sample of $n$ iid data-points from the unimodal distribution. We also refer the readers to \Cref{app:auxi} to see another variant of the algorithm (using \eqref{eq:edelman}) which can provide finite sample valid confidence set of the mode even when the data-points are not independent. 
    \item We propose and discuss a general method of constructing a finite sample valid confidence set of the mode of a $d$-dimensional $\gamma$-unimodal distribution using any algorithm that provides finite sample valid confidence sets of the mode of a unimodal univariate distribution. 
\end{enumerate}

\paragraph{{\bf Organization}} 
In \Cref{sec:methods} we introduce different methods of constructing finite sample valid confidence sets of the mode of a unimodal distribution and study the behavior of the width of these confidence sets with increasing sample size. In \Cref{sec:simulation} we investigate the performance of the methods developed in this work on simulated data. The proofs of all the main results along with auxiliary results are presented in the appendix. 

%The usual Wald confidence interval of median (\eqref{eq:as_norm_cn}) does not ensure the desired coverage for all sample sizes. Moreover, it requires the differentiability of the distribution function at the median and the boundedness of the derivative of the distribution function at the median away from zero. The same assumptions are required for the classical bootstrap based confidence interval.

\section{Methods}
\label{sec:methods}  
Suppose $X_1, \cdots, X_n \stackrel{iid}{\sim} F$ where $F$ is a unimodal univariate distribution function whose mode is $\theta_0$. A confidence set $\widehat{\mathrm{CS}}_{n,\alpha}$ is said to be a finite sample valid confidence set of $\theta_0$ at level $1 - \alpha$ (for $0 < \alpha <1 $) if, 
\[
\mathbb{P}_F(\theta_0 \in \widehat{\mathrm{CS}}_{n,\alpha}) \geq 1 - \alpha \quad \mbox{for all } n \in \mathbb{N}. 
\]
In this section we shall introduce and discuss three distinct approaches of constructing finite sample valid confidence set of mode: (1) Using the spacings between the order statistics, (2) Using M-estimation, (3) Using \cite{edelman1990confidence}'s result \eqref{eq:edelman}. We shall prove the finite sample validity and derive the rate of the width of the confidence sets computed through each of these methods.
\subsection{A nested approach of obtaining confidence interval of mode using order statistics}
\label{subsec:os_approach}
In this sub-section we discuss how we can use the spacings between order statistics to construct finite sample valid confidence set of the mode. In \Cref{alg:proposed-conf-int_os} we describe the algorithm for computing a finite sample confidence set of $\theta_0$. 
\begin{theorem}
\label{thm:main_result_os_validity}
    Suppose $X_1, X_2, \ldots, X_n$ are independent and identically distributed as $F$ which is a unimodal distribution with mode $\theta_0$ i.e.\ $F(\cdot)$ is a convex function in the interval $\left(-\infty, \theta_0 \right]$ and a concave function in the interval $\left[ \theta_0, \infty \right)$. Then the confidence interval returned by~\Cref{alg:proposed-conf-int_os} satisfies the following:
    \[
    \mbox{For all $n\ge1$ and for any $\alpha\in(0, 1)$, $\mathbb{P}_F(\theta_0\in\widehat{\mathrm{CI}}_{n,\alpha}) \ge 1 - \alpha$.}
    \]
\end{theorem}

\begin{algorithm}
    \caption{\textbf{(M1)} Proposed finite sample valid confidence set for the population mode}
    \label{alg:proposed-conf-int_os}
    \KwIn{Sample: $X_1,\cdots,X_n$ and Confidence Level: $1-\alpha$}
    \KwOut{Confidence set for the population mode with finite sample coverage}
    Define the following quantities,
    \begin{equation*}
        \begin{split}
            s_n &= \lceil\log_2 \log (n) \rceil, \\
           B_{max} &= \left \lfloor \log_2 (n/8) \right \rfloor - s_n, \\
           n_B &= \left \lfloor \frac{n-1}{2^{B+s_n}} \right \rfloor \quad \mbox{for $0 \leq B \leq B_{max}$},\\
           t_n &= \sum_{B=0}^{B_{max}} \frac{1}{B+2}, \\
           h_B &= \frac{q\mbox{Beta}(1 -(\alpha/(4(B+2)n_Bt_n)),2^{B+s_n},n+1 - 2^{B+s_n})}{q\mbox{Beta}(\alpha/(4(B+2)n_Bt_n),2^{B+s_n},n+1 - 2^{B+s_n})} \quad \mbox{for $0 \leq B \leq B_{max}$}. 
           %I_B &= \left\{(j,k) : j = 1 +(i-1)2^B, k = 1 + i 2^B \mbox{ for } i = 1,\cdots, n_B \right\} \quad \mbox{for } 0 \leq B \leq B_{max}.
        \end{split}
    \end{equation*}\\
    %Define the following set of order statistics,
    %\[
    %x_i = X_{(1+(i-1)2^{s_n})} \quad \mbox{for } i=1,2,\cdots, m = \left \lfloor \frac{n-1}{2^{s_n}} \right \rfloor +1.
    %\] \\
    For $0 \leq B \leq B_{max}$ we define the following intervals constructed by the order statistics,
    \[
    I_{Bi} = [X_{(1+(i-1)2^{B+s_n})},X_{(1+i2^{B+s_n})}] \quad \mbox{for } i=1,\cdots,n_B.
    \] \\
    Start with $B = B_{max}$. Let $I_{B_{max}i_{min}} = \mbox{arg min}_{1 \leq i \leq n_{B_{max}}} \mbox{Width}(I_{B_{max}i})$. Obtain the largest set of neighbouring intervals of $I_{B_{max}i_{min}}$ i.e.\ a set of the form $K_{B_{max}} = \{I_{B_{max}(\max\{i_{min}-k,1\})},\cdots,I_{B_{max}(i_{min}-1)},I_{B_{max}i_{min}},I_{B_{max}(i_{min}+1)},\cdots,I_{B_{max}(\min\{i_{min}+l,n_{B_{max}}\})}\}$ such that the following holds,
    \[
    \mbox{Width}(I_{B_{max}i}) \leq h_{B_{max}} \mbox{Width}(I_{B_{max}i_{min}}) \quad \mbox{for } \max\{i_{min}-k,1\} \leq i \leq \min\{i_{min}+l,n_{B_{max}}\}.
    \] \\
    Repeat the following steps as $B$ goes down from $B_{max}-1$ to $0$. Let $I_{Bi_{min}} = \mbox{arg min}_{I_{Bi} \in \{K_{B+1} \cap \{I_{B1},\cdots,I_{Bn_B} \}\}} \mbox{Width}(I_{Bi})$. Obtain the largest set of neighbouring intervals of $I_{Bi_{min}}$ from the collection of intervals $\{K_{B+1} \cap \{I_{B1},\cdots,I_{Bn_B} \}\}$ i.e.\ a set of the form $K_B = \{I_{B(\max\{i_{min}-k,1\})},\cdots,I_{B(i_{min}-1)},I_{Bi_{min}},I_{B(i_{min}+1)},\cdots,I_{B(\min\{i_{min}+l,n_B\})}\}$ such that the following holds,
    \[
    \mbox{Width}(I_{Bi}) \leq h_{B} \mbox{Width}(I_{Bi_{min}}) \quad \mbox{for } \max\{i_{min}-k,1\} \leq i \leq \min\{i_{min}+l,n_{B}\}.
    \] 
    After going through the above steps we obtain $K_0$. \\
    Return the following confidence set,
    \begin{equation*}
    \begin{split}
    \widehat{\mathrm{CI}}_{n,\alpha} = & \bigcup_{I_{0i} \in K_0}I_{0i} \bigcup \left\{\left(X_{(1)} - \lambda_{n,\alpha}(X_{(n)}-X_{(1)}), X_{(1)} \right] \bm 1 \{I_{01} \in K_0 \} \right\} \\
    &\bigcup \left\{\left[X_{(n)}, X_{(n)} + \lambda_{n,\alpha}(X_{(n)}-X_{(1)})\right) \bm 1 \{I_{0n_0} \in K_0 \} \right\},
    \end{split}
    \end{equation*}
    where $\lambda_{n,\alpha} = (\alpha/2)^{-1/(n-1)} -1$.
\end{algorithm}
\Cref{thm:main_result_os_validity} establishes the finite sample validity of the confidence set proposed in \Cref{alg:proposed-conf-int_os}. It is important to note that by construction the confidence set $\widehat{\mathrm{CI}}_{n,\alpha}$ is an interval. The proof of the theorem is based on the fundamental property of unimodal distribution functions: for two adjacent disjoint intervals $I_1 = [a,b],I_2 = [b,c]$ (where $-\infty<a<b<c<\infty$) if,
        \[
            F(I_1)/|I_1| > F(I_2)/|I_2|,
        \]
then either the mode $\theta_0 \in I_1$ or $\theta_0 <a$ i.e.\ the mode is ``closer" to the interval $I_1$ as compared to $I_2$. Thus we can drop $I_2$ and the intervals present after $I_2$ from the collection of intervals that can contain the mode. We use this idea by replacing $a , b, c$ with different order statistics. It is possible to obtain a set of distribution functions $\mathcal{C}_n(\alpha)$ which is collection of all $F$ such that 
            \[
         c_B \leq F(I_{Bi}) \leq d_B, \quad \mbox{($I_{Bi}$ are spacings of different types)},
            \]
            for $i=1,\cdots,n_B$ and $B = 0,\cdots,B_{max}$ such that, 
            \[
                \mathbb{P}_F(F \in \mathcal{C}_n (\alpha)) \geq 1-\alpha.    
            \] 
The proof is completed by noting that valid coverage of the distribution function implies the valid coverage of the functional, mode. For the detailed proof, refer to \Cref{app:thm_os_valid}. 

For analyzing the width of this confidence interval, we assume that the distribution function $F$ is absolutely continuous with associated density $f$. We further assume that there exists an interval $[c, d] \subset (-\infty, \infty)$ such that for each open set $U$ containing $\theta_0$, there exists $\epsilon = \epsilon(U) >0$ such that $f(x) + \epsilon \leq f(\theta_0)$ for each $x \in [c,d] - U$. For $c \leq \theta - R_1\delta < \theta - \delta < \theta$ and $\theta < \theta + \delta < \theta + R_2\delta \leq d$, we consider the following definitions stated in \cite{sager1975consistency}, 
\begin{equation}
\label{eq:alpha_def}
    \begin{split}
        r^-(\delta) &= \inf\{ f(x) | \theta \leq x \leq \theta + \delta \}, \\
        r^+(R_2 \delta) &= \sup\{ f(x) | \theta + R_2\delta \leq x \leq d \},\\
        l^-(\delta) &= \inf\{ f(x)| \theta- \delta \leq x \leq \theta \}, \\
        l^+(R_1\delta) &= \sup\{ f(x) |c \leq x \leq \theta - R_1\delta \}, \\
         r(\delta, R_2\delta) &= \frac{r^-(\delta)}{r^+(R_2\delta)}, \\
        l(\delta, R_1\delta) &= \frac{l^-(\delta)}{l^+(R_1\delta)},\\
        \alpha(\delta, R_1,R_2) &= \frac{\min\{r^-(\delta), l^-(\delta) \}}{\max\{r^+(R_2 \delta), l^+(R_1\delta) \}}. 
    \end{split}
\end{equation}
\begin{theorem}
    \label{thm:width_analysis_os}
    Let $F$ be an absolutely continuous unimodal distribution function with mode $\theta_0 $ and density $f$. We assume that there exists an interval $[c, d] \subset (-\infty, \infty)$ such that for each open set $U$ containing $\theta_0$, there exists $\epsilon = \epsilon(U) >0$ such that $f(x) + \epsilon \leq f(\theta_0)$ for each $x \in [c,d] - U$. Suppose the following condition holds true, 
    % \begin{enumerate}
    %     \item $r(\delta, R\delta) \geq 1 + \rho\delta^k$ for all small positive $\delta$ if $\theta_0 = c$. 
    %     \item $l(\delta, R\delta) \geq 1 + \rho \delta^k$ for all small positive $\delta$ if $\theta_0 = d$. 
    %     \item $\alpha(\delta, R_1, R_2) \geq 1 + \rho \delta^k$ for all small positive $\delta$ if $\theta_0 \in (c,d) $.
    % \end{enumerate}
    \[
    \alpha(\delta, R_1, R_2) \geq 1 + \rho \delta^{\beta}, \quad (\alpha(\cdot, \cdot, \cdot)\mbox{ is as defined in \eqref{eq:alpha_def}}),
    \]
    for all small positive $\delta$ and some $R_1,R_2 >1$, and $\rho, \beta > 0$. Then we have the following, 
    \[
    n^{1/(1+2\beta)}(\log n)^{-1/\beta}\mathrm{Width}(\widehat{\mathrm{CI}}_{n,\alpha}) \stackrel{P}{\rightarrow} 0 \mbox{ as } n \rightarrow \infty.
    \]
\end{theorem}
\Cref{thm:width_analysis_os} demonstrates that $\mathrm{Width}(\widehat{\mathrm{CI}}_{n,\alpha})$ shrinks to $0$ at the rate of $n^{-1/(1+2\beta)}(\log n)^{1/\beta}$ without any knowledge of the regularity parameter $\beta$. The assumption on $\alpha(\cdot, \cdot, \cdot)$ basically means that near the mode, the density behaves like a power function with exponent $\beta$. The main idea of the proof of the theorem is to compare the true weights $F(I_n)$ with the empirical weights $F_n(I_n)$ where $I_n$ is an interval containing $r(n)$ many observations. For the detailed proof, refer to \Cref{app:thm_os_width}. Under assumption~\ref{assump:f}, \cite{arias2022estimation} shows that the minimax rate of convergence of an estimate of mode is $n^{-1/(1 + 2\beta)}$ if the parameter $\beta$ governing the regularity of the density about the mode is known a priori.  
\begin{assumption}[Behavior of $f$ about the mode]
\label{assump:f}
$f$ has a unique mode at $\theta_0$ and for some $0<c_0<C_0$, $h_0 >0$, and $\beta > 0$, 
\begin{equation*}
    \begin{split}
        f(\theta_0) - C_0|\theta - \theta_0|^{\beta} &\leq f(\theta) \leq f(\theta_0) - c_0|\theta - \theta_0|^{\beta} , \quad \mbox{when } |\theta - \theta_0| \leq h_0, \\
        f(\theta) &\leq f(\theta_0) - c_0h^{\beta}, \quad \mbox{when } |\theta - \theta_0| \geq h_0.
    \end{split}
\end{equation*}
\end{assumption}
\begin{remark}
Note that assumption~\ref{assump:f} is a stricter assumption than that made in \Cref{thm:width_analysis_os}. In particular if assumption~\ref{assump:f} holds true and $\alpha(\cdot, \cdot, \cdot)$ is as defined in \eqref{eq:alpha_def} then it can be easily checked that for any $0<\delta < 2 \delta < h_0$, $\alpha(\delta, 2, 2) \geq 1 + \rho \delta^{\beta}$ for a suitable $\rho > 0$. Thus \Cref{thm:width_analysis_os} establishes that under assumption~\ref{assump:f}, $\mathrm{Width}(\widehat{\mathrm{CI}}_{n,\alpha})$ shrinks to $0$ at the minimax optimal rate upto a logarithmic factor. The departure from the exact minimax rate of $n^{-1/(1 + 2\beta)}$ can be attributed to the adaptivity (see for instance \cite{klemela2005adaptive}) which this confidence interval provides ($\widehat{\mathrm{CI}}_{n,\alpha}$ does not require the prior knowledge of $\beta$). 
\end{remark}

\subsection{Confidence set of mode using M-estimation}
\label{subsec:m-estimation}
It is possible to obtain valid confidence set of the mode of a unimodal distribution using M-estimation procedure. \Cref{alg:M_proposed-conf-int} illustrates the steps for constructing such a confidence set. If $S \subset \mathbb{R}$ we use $S \pm h$ to denote the set $\{x \in \mathbb{R}: x - h \in S \mbox{ or } x + h \in S\}$. 
\begin{algorithm}
    \caption{\textbf{(M2)} Proposed finite sample valid confidence set for the population mode}
    \label{alg:M_proposed-conf-int}
    \KwIn{Sample: $X_1,\cdots,X_m$ and Confidence Level: $1-\alpha$, and tuning parameter $h > 0$}
    \KwOut{Confidence set for the population mode with finite sample coverage}
    Divide the sample into two almost equally sized disjoint sub-samples $S_1$ and $S_2$. Let $|S_1| = m - n$ and $|S_2| = n  $. \\
    Compute any consistent estimate $\widehat \theta_1$ of mode based on the sub-sample $S_1$. \\
    Consider the function $m_{\theta;h}(X) = -(1/2h)\textbf{1}\{\theta - h < X \leq \theta + h \}$. \\
 Return $\widehat{\mathrm{CS}}_{n, \alpha}^h$, 
 \[
 \widehat{\mathrm{CS}}_{n, \alpha}^h = \left\{\theta \in \mathbb{R} \Bigg| P_{n}(m_{\theta; h} - m_{\widehat{\theta}_1;h}) \leq \frac{1}{h}\sqrt{\frac{3}{2n}}\left[\sqrt{\log(1/\alpha)} + 2 \right] \right\} \pm h,
 \]
 where the average $P_n$ is taken over the second sub-sample $S_2$.    
\end{algorithm}
\begin{theorem}
\label{thm:m_estimation}
Suppose $X_1, \cdots, X_m$ are independent and identically distributed as $F$ (density $f$) which is an absolutely continuous unimodal distribution with a unique mode $\theta_0$. Then the confidence interval returned by \Cref{alg:M_proposed-conf-int} satisfies the following for all $n\ge1$, for any $h > 0 $, and for any $\alpha\in(0, 1)$,
  \[
    \mathbb{P}_F(\theta_0\in\widehat{\mathrm{CS}}_{n,\alpha}^{h}) \ge 1 - \alpha. 
    \]
\end{theorem}
\Cref{thm:m_estimation} shows the finite sample validity of the confidence interval proposed in \Cref{alg:M_proposed-conf-int}. We give a brief outline of the proof of the theorem here. If $\theta_h$ denotes the mode of the convolution distribution $F*\mbox{Uniform}[-h,h]$, then it can be shown that $\theta_0 \in (\theta_h - h, \theta_h + h)$. Thus finding a valid confidence set of $\theta_0$ boils down to finding a valid confidence set of $\theta_h$. It is easy to check that the density function of the convolution $F*\mbox{Uniform}[-h,h]$ is $(F(x+h) - F(x - h))/(2h)$ and hence constructing confidence set of the mode $\theta_h$ of $F*\mbox{Uniform}[-h,h]$ can be viewed as an M-estimation problem, 
\[
\theta_h = \mbox{arg max}_x (F(x+h) - F(x - h))/(2h) = \mbox{arg min}_x \mathbb{E}\left[-\frac{1}{2h} \textbf{1}\{x - h < X \leq x + h\}\right]. 
\]
Given this framework, confidence set of $\theta_h$ can be obtained using results from \cite{takatsu2025bridging}. The detailed proof of \Cref{thm:m_estimation} is provided in \Cref{app:m_validity}. 
\begin{remark}
We have used Hoeffding type bounds to construct the confidence set of $\theta_h$ in \Cref{alg:M_proposed-conf-int}. It is possible to obtain better (of smaller width) confidence set by using tighter bounds (such as Bernstein) in the M-estimation problem. 
\end{remark}
We study the width of the confidence set $\widehat{\mathrm{CS}}_{n,\alpha}^{h}$ under assumption~\ref{assump:f}. We show that if the the estimator $\widehat \theta_1$ computed on $S_1$ is minimax optimal then for appropriate choice of the tuning parameter $h$, $\mathrm{Width}(\widehat{\mathrm{CS}}_{n,\alpha}^{h})$ converges to $0$ at the minimax optimal rate (upto some logarithmic factor).  
\begin{theorem}
\label{thm:m_estimator_width}
Suppose $X_1, \cdots, X_{2n}$ are independent and identically distributed as $F$ (density $f$) which is an absolutely continuous unimodal distribution with a unique mode $\theta_0$ and which satisfies assumption~\ref{assump:f} with $\beta \geq 1$. Let $\widehat \theta_1$ be a minimax optimal estimator of $\theta_0$ i.e.\ $n^{1/(1 + 2\beta)}( \widehat \theta_1 - \theta_0) = O_P(1)$. Let $\ell(n)$ be any sequence that satisfies $\ell(n) \rightarrow \infty$ and $\ell(n) = o(n^{\tau})$ as $n \rightarrow \infty$ for any $\tau > 0$. Then for $h = n^{-1/(1 + 2\beta)}\ell(n)^{1/2}$ we have, 
\[
n^{1/(1 + 2\beta)} \ell(n)^{-1}\mathrm{Width}(\widehat{\mathrm{CS}}_{n,\alpha}^{h}) = O_P(1) \quad \mbox{as}\quad n \rightarrow \infty. 
\]
\end{theorem}
\Cref{thm:m_estimator_width} establishes that for appropriate choice of $h$, the rate of $\mathrm{Width}(\widehat{\mathrm{CS}}_{n,\alpha}^{h})$ departs from the minimax optimal rate $n^{-1/(1 + 2\beta)}$ by a factor of $\ell(n)$ which can be allowed to diverge with $n$ at an arbitrarily slow rate. The proof of the rate of convergence of $\mathrm{Width}(\widehat{\mathrm{CS}}_{n,\alpha}^{h})$ is based on the standard methods used for establishing convergence rates of M-estimators (see for instance Theorem $3.2.5$ of \cite{wellner2013weak}). The complete proof of \Cref{thm:m_estimator_width} is provided in \Cref{app:proof_m_width}. 
\begin{remark}
A consequence of the proof of \Cref{thm:m_estimator_width} is that if we let $h = \log(n)$ then $\mathrm{Width}(\widehat{\mathrm{CS}}_{n,\alpha}^{h})$ shrinks to $0$ (not at the optimal rate). However as noted earlier, for the width of the confidence set $\widehat{\mathrm{CS}}_{n,\alpha}^{h}$ to shrink to $0$ at the near minimax optimal rate, proper choice of the tuning parameter $h$ is of paramount importance. In the next section we propose a completely adaptive confidence set of mode (using M-estimation) whose width shrinks to $0$ near-optimally without any need for tuning the bandwidth $h$.  
\end{remark}

\subsection{An adaptive confidence set of mode using M-estimation}
\label{subsec:ada_M}
In this sub-section we propose a modification of the confidence set of mode obtained through M-estimation procedure (\Cref{alg:M_proposed-conf-int}) so that proper choice of tuning parameter $h$ no longer affects the optimality of the width of the confidence set. \Cref{alg:ada_M} describes the construction of this adaptive confidence set. 
\begin{algorithm}
    \caption{\textbf{(M2')} Proposed finite sample valid confidence set for the population mode}
    \label{alg:ada_M}
    \KwIn{Sample: $X_1,\cdots,X_m$ and Confidence Level: $1-\alpha$}
    \KwOut{Confidence set for the population mode with finite sample coverage}
    Divide the sample into two almost equally sized disjoint sub-samples $S_1$ and $S_2$. Let $|S_1| = m - n$ and $|S_2| = n  $. \\
    Compute any consistent estimate $\widehat \theta_1$ of mode based on the sub-sample $S_1$. \\
    Consider the function $m_{\theta;h}(X) = -(1/2h)\textbf{1}\{\theta - h < X \leq \theta + h \}$ for any $h > 0$. \\
 Return $\widehat{\mathrm{CS}}_{n, \alpha}$, 
 \[
    \widehat{\mathrm{CS}}_{n, \alpha} = \widehat{\mathrm{CS}}_{n, \alpha}^{\widehat{h}}\quad \mbox{where}\quad \widehat{h} := \mbox{arg min}_{h > 0} \left\{\mbox{Width}\left( \widehat{\mathrm{CS}}_{n, \alpha}^h\right) \right\} \quad \mbox{with},
 \]
 \[
 \widehat{\mathrm{CS}}_{n, \alpha}^h = \left\{\theta \in \mathbb{R} \Bigg| P_{n}(m_{\widehat \theta_1; h} - m_{\widehat \theta_1;h}) \leq \frac{1}{h} \sqrt{\frac{2\log(2/\alpha)}{n}}  \right\} \pm h.
 \]
In the above definition, the average $P_n$ is taken over the second sub-sample $S_2$.    
\end{algorithm}
\begin{theorem}
    \label{thm:ada_m_estimation}
    Suppose $X_1, \cdots, X_m$ are independent and identically distributed as $F$ (density $f$) which is an absolutely continuous unimodal distribution with a unique mode $\theta_0$. Then the confidence interval returned by \Cref{alg:ada_M} satisfies the following for all $n\ge1$, and for any $\alpha\in(0, 1)$,
      \[
        \mathbb{P}_F(\theta_0\in\widehat{\mathrm{CS}}_{n,\alpha}) \ge 1 - \alpha. 
        \]
    \end{theorem}
The proof of \Cref{thm:ada_m_estimation} is very similar to that of \Cref{thm:m_estimation}. We find a simultaneously valid confidence set of $\theta_h$ for all $h > 0$ and use that to obtain a valid confidence set of $\theta_0$. The detailed proof is provided in \Cref{app:ada_M_estimation}. 

As in the earlier sections we analyze the width of the confidence set $\widehat{\mathrm{CS}}_{n, \alpha}$ under assumption~\ref{assump:f} and show that using a minimax optimal estimator $\widehat \theta_1$ yields an optimal rate for $\mathrm{Width}(\widehat{\mathrm{CS}}_{n,\alpha})$. 
\begin{theorem}
\label{thm:ada_m_estimator_width}
Suppose $X_1, \cdots, X_{2n}$ are independent and identically distributed as $F$ (density $f$) which is an absolutely continuous unimodal distribution with a unique mode $\theta_0$ and which satisfies assumption~\ref{assump:f} with $\beta \geq 1$. Let $\widehat \theta_1$ be a minimax optimal estimator of $\theta_0$ i.e.\ $n^{1/(1 + 2\beta)}( \widehat \theta_1 - \theta_0) = O_P(1)$. Let $\ell(n)$ be any sequence that satisfies $\ell(n) \rightarrow \infty$ and $\ell(n) = o(n^{\tau})$ as $n \rightarrow \infty$ for any $\tau > 0$. Then we have, 
\[
n^{1/(1 + 2\beta)} \ell(n)^{-1}\mathrm{Width}(\widehat{\mathrm{CS}}_{n,\alpha}) = O_P(1) \quad \mbox{as}\quad n \rightarrow \infty. 
\]
\end{theorem}
\Cref{thm:ada_m_estimator_width} establishes that the rate of $\mathrm{Width}(\widehat{\mathrm{CS}}_{n,\alpha})$ departs from the minimax optimal rate $n^{-1/(1 + 2\beta)}$ by a factor of $\ell(n)$ which can be allowed to diverge with $n$ at an arbitrarily slow rate. As noted earlier in \Cref{subsec:os_approach}, the departure from the minimax optimal rate can be attributed to the adaptivity (yields ideal rate without any prior knowledge of $\beta$) which this confidence set provides. The proof of the rate of $\mathrm{Width}(\widehat{\mathrm{CS}}_{n,\alpha})$ re-traces the steps of the proof of \Cref{thm:m_estimator_width} and the details of the proof can be seen in \Cref{app:ada_M_width}. 

\subsection{Confidence set of mode using Edelman's result}
\label{subsec:edelman_res_method}
In this sub-section we use \cite{edelman1990confidence}'s result \eqref{eq:edelman} to construct a finite sample valid confidence set of the mode. The width of this confidence set does not shrink to $0$ as the sample size increases. Despite this, we cover this method because it is easy to understand. The steps of construction of this confidence set is described in \Cref{alg:proposed-conf-int_ed}.

\begin{algorithm}
    \caption{\textbf{(M3)} Proposed finite sample valid confidence set for the population mode}
    \label{alg:proposed-conf-int_ed}
  \KwIn{Sample: $X_1,\cdots,X_m$ and Confidence Level: $1-\alpha$}
    \KwOut{Confidence set for the population mode with finite sample coverage}
    Divide the sample into two almost equally sized disjoint sub-samples $S_1$ and $S_2$. Let $|S_1| = m - n$ and $|S_2| = n  $. \\
    Compute any consistent estimate $\widehat \theta_1$ of mode based on the sub-sample $S_1$. \\
    For any $\theta \in \mathbb{R}$ let, 
    \[
    p_i(\theta) = \frac{2}{1 + \left|\frac{X_i - \theta}{X_i - \widehat \theta_1} \right|} \mbox{ for all } X_i \in S_2. 
    \] \\
    Return the confidence set, 
    \[
    \widehat{\mathrm{CS}}_{n, \alpha}^{\mathrm{Ed}_p} = \left\{\theta \in \mathbb{R}\Big| -2\sum_{i: X_i \in S_2}\log(p_i(\theta)) < \chi^2( 1- \alpha, 2n) \right\}. 
    \]
    \end{algorithm}
\begin{theorem}
    \label{thm:ed_method_validity}
Suppose $X_1, \cdots, X_m$ are independent and identically distributed as $F$ (density $f$) which is an absolutely continuous unimodal distribution with a unique mode $\theta_0$. Then the confidence set returned by \Cref{alg:proposed-conf-int_ed} satisfies the following for all $n\ge1$ and for any $\alpha\in(0, 1)$,
  \[
    \mathbb{P}_F(\theta_0\in \widehat{\mathrm{CS}}_{n, \alpha}^{\mathrm{Ed}_p}) \ge 1 - \alpha.
    \]
\end{theorem}
\Cref{thm:ed_method_validity} establishes the finite sample validity of the confidence set proposed in \Cref{alg:proposed-conf-int_ed}. The theorem is proved by applying Edelman's result \eqref{eq:edelman} to all the points in the second sub-sample $S_2$ and showing that $p_i(\theta_0)$ is a valid p-value for all $X_i \in S_2$ i.e.\ $\mathbb{P}(p_i(\theta_0) \leq \alpha) \leq \alpha$ for any $\alpha \in (0,1)$ and all $X_i \in S_2$. Refer to \Cref{app:proof_ed_validity} for the detailed proof. 

We analyze the width of $\widehat{\mathrm{CS}}_{n, \alpha}^{\mathrm{Ed}_p}$ under the assumption that $\mathbb{E}_F[\log(1 + |(X-\theta)/(X - \theta_0)|] $ exists for all $\theta \in \mathbb{R}$ and show that unlike the confidence sets discussed in the previous sections, the width of $\widehat{\mathrm{CS}}_{n, \alpha}^{\mathrm{Ed}_p}$ does not converge to $0$ as $n$ increases. 
\begin{theorem}
\label{thm:ed_width}
Suppose $X_1, \cdots, X_m$ are independent and identically distributed as $F$ (density $f$) which is an absolutely continuous unimodal distribution with a unique mode $\theta_0$ such that $\mathbb{E}_F[\log(1 + |(X-\theta)/(X - \theta_0)|)] $ exists for all $\theta \in \mathbb{R}$. Let $\widehat \theta_1$ be a consistent estimator of $\theta_0$. Then we have, 
\[
\widehat{\mathrm{CS}}_{n, \alpha}^{\mathrm{Ed}_p} \stackrel{a.s.}{\rightarrow} \left\{\theta \in \mathbb{R} \Bigg| \mathbb{E}_F \left[ \log\left(1 +  \left|\frac{X - \theta}{X - \theta_0} \right| \right)\right] < 1+ \log 2\right\} \quad \mbox{as} \quad m-n,n \rightarrow \infty,
\]
\end{theorem}
The proof of \Cref{thm:ed_width} follows by an application of the law of large numbers. For the complete proof, refer to \Cref{app:proof_ed_width}. \Cref{thm:ed_width} states that as the sample size increases $\widehat{\mathrm{CS}}_{n, \alpha}^{\mathrm{Ed}_p}$ shrinks to a limiting set and hence the width of $\widehat{\mathrm{CS}}_{n, \alpha}^{\mathrm{Ed}_p}$ stays bounded away from $0$. 

In \Cref{app:auxi} we provide an alternative method of constructing confidence set for the mode based on Edelman's result \eqref{eq:edelman} that stays valid even when $X_1, \cdots, X_m$ are not independent.  
\subsection{Extension to $\gamma$-Unimodal (multivariate) distributions}
\label{subsec:gamma_unimodal_distr}
In this sub-section we demonstrate the process of  obtaining confidence sets of the mode of $\gamma$-unimodal (multivariate) distribution utilizing the algorithms developed in the earlier sub-sections. Before going into $\gamma$-unimodal distributions, we state a couple of properties by which we can characterize unimodal univariate distributions. 
\begin{result}[Theorem 1.3 \cite{dharmadhikari1988unimodality}]
\label{res:uz_unimode}
A distribution function $F$ on $\mathbb{R}$ is unimodal about $0$ iff there exists independent random variables $U$ and $Z$ such that $U$ is uniform on $(0,1)$ and the product $UZ$ has distribution function $F$.
\end{result}
\begin{result}[Theorem 1.4 \cite{dharmadhikari1988unimodality}]
\label{res:g_unimode}
A random variable $X$ has a unimodal distribution about $0$ iff $t\mathbb{E}[g(tX)]$ is non-decreasing in $t> 0$ for every bounded, nonnegative,
Borel-measurable function $g$. 
\end{result}
The characterizing property in \Cref{res:g_unimode} is used to define $\gamma$-unimodal distributions (see \cite{olshen1970generalized}, \cite{kolyvakis2023multivariate}), 
\begin{definition}
 \label{def:alpha_unimodal}  
A random $d$-dimensional vector $X \in \mathbb{R}^d$ is said to have a $\gamma$-unimodal distribution ($\gamma > 0$)  about 0 if, for every bounded, nonnegative, Borel measurable function $g: \mathbb{R}^d \rightarrow \mathbb{R}$ the quantity $t^{\gamma} \mathbb{E}[g(tX)]$ is non-decreasing in $t \in (0, \infty)$. 
\end{definition}
The class of $\gamma$-Unimodal distributions about $0$ can be characterized by a similar property as that stated in \Cref{res:uz_unimode}. 
\begin{result}[\cite{olshen1970generalized}]
\label{res:uz_gamma_modal}
A random $d$-dimensional vector $X \in \mathbb{R}^d$ has $\gamma$-unimodal distribution about $0$ iff $X$ is distributed as $U^{1/\gamma}Z$ where $U$ is uniform on $(0,1)$ and $Z$ is independent of $U$. 
\end{result}
To state the next property of $\gamma$-unimodal distributions, which we shall use to construct valid confidence sets of mode, we need the definition of a Minkowski functional. The Minkowski functional $\pi_S(x):\mathbb{R}^d \rightarrow \mathbb{R}$ of the set $S \subset \mathbb{R}^d$ is defined as, 
\[
\pi_S(x) \coloneqq \inf\{a > 0: x \in aS\} \quad \forall \mbox{ } x \in \mathbb{R}^d. 
\]
Note that if we take $S = \{x \in \mathbb{R}^d: \|x\|_2 \leq 1\}$, then $\pi_S(x) = \|x\|_2$. Thus the norm $x \rightarrow \|x\|_2$ is an example of a Minkowski functional. The following property of Minkowski functionals can be checked easily. 
\begin{result}[Proposition 2.3 \cite{dasgupta1995new}]
    \label{res:minkowski}
Given a star-shaped set $S \subset \mathbb{R}^d$ about $0$ (i.e. $x \in S$ implies that all the points in the line segment joining $x$ and $0$ are in $S$), the Minkowski functional $\pi_S(\cdot)$ of the set $S$ is homogeneous of degree $1$ i.e.
\[
\pi_S(cx) = c \pi_S(x) \quad \mbox{for all $c \geq 0$, for all $x \in \mathbb{R}^d$}. 
\]
\end{result}
Using \Cref{res:minkowski} we show the following property of $\gamma$-unimodal distributions. 
\begin{lemma}
\label{lem:minkow_gamma}
    Suppose $X \in \mathbb{R}^d$ follows a $\gamma$-unimodal distribution about $\theta_0 \in \mathbb{R}^d$. Consider a star-shaped set $S \subset \mathbb{R}^d$ about $0$ and the corresponding Minkowski functional $\pi_S(\cdot)$. Then $[\pi_S(X - \theta_0)]^{\gamma}$ follows a unimodal distribution about $0$ on $\mathbb{R}$. 
\end{lemma}
\begin{proof}[Proof of \Cref{lem:minkow_gamma}]
Since $X$ follows a $\gamma$-unimodal distribution about $\theta_0$, $X - \theta_0$ follows $\gamma$-unimodal distribution about $0$. By \Cref{res:g_unimode}, there exists a uniform-$(0,1)$ random variable $U$ and a random vector $Z \in \mathbb{R}^d$ independent of $U$ such that $X - \theta_0 = U^{1/\gamma} Z$. From \Cref{res:minkowski} we know that $\pi_S(\cdot)$ is homogeneous of degree $1$. This implies that $\pi_S(X - \theta_0) = \pi_S(U^{1/\gamma} Z) = U^{1/\gamma} \pi_S(Z) $. In other words, $[\pi_S(X - \theta_0)]^{\gamma} = U [\pi_S(Z)]^{\gamma}$ where $U$ is a uniform-$(0,1)$ random variable independent of the random variable $[\pi_S(Z)]^{\gamma}$. Using the characterization in \Cref{res:uz_unimode}, we conclude that $[\pi_S(X - \theta_0)]^{\gamma}$ follows unimodal distribution about $0$ on $\mathbb{R}$. 
\end{proof}
An immediate consequence of \Cref{lem:minkow_gamma} is that if $X$ follows a $\gamma$-unimodal distribution about $\theta_0$, then $\|X - \theta_0\|_2^{\gamma}$ follows a unimodal univariate distribution on $\mathbb{R}$ with mode $0$. Thus we can apply the algorithms for finding the mode of a unimodal univariate distribution (\Cref{alg:proposed-conf-int_os}, \Cref{alg:M_proposed-conf-int}, \Cref{alg:ada_M}, \Cref{alg:proposed-conf-int_ed}) to the transformed data points $\{\|X_i - \theta_0\|_2^{\gamma} \}_{i = 1}^n$ to obtain valid confidence sets for $\theta_0$. We illustrate this algorithm in \Cref{alg:gamma-conf-int_ed}. 
\begin{algorithm}
    \caption{\textbf{(M4)} Proposed finite sample valid confidence set for the population mode of a $d$-dimensional $\gamma$-unimodal distribution}
    \label{alg:gamma-conf-int_ed}
    \KwIn{Sample: $X_1,\cdots,X_n$, Confidence Level: $1-\alpha$, an algorithm $\mathcal{A}(\{y_i\}_{i = 1}^n, 1 - \alpha)$ for constructing a valid $(1-\alpha)$ confidence set of the mode of the univariate data $\{y_i\}_{i = 1}^n$}
    \KwOut{Confidence set for the population mode with finite sample coverage}
   For any $\theta \in \mathbb{R}$ let, 
   \[
   \widehat{\mathrm{CS}}_{n,\alpha}^{\theta; \gamma} = \mathcal{A}(\{\|X_i - \theta\|_2^{\gamma}\}_{i = 1}^n, 1 - \alpha) . 
   \]
    %Define the following set of order statistics,
    %\[
    %x_i = X_{(1+(i-1)2^{s_n})} \quad \mbox{for } i=1,2,\cdots, m = \left \lfloor \frac{n-1}{2^{s_n}} \right \rfloor +1.
    %\] \\
     \\
 Return the confidence set, 
    \begin{equation*}
    \begin{split}
   \widehat{\mathrm{CS}}_{n,\alpha}^{ \gamma} =& \left\{\theta \in \mathbb{R}^d\Bigg|  0 \in \widehat{\mathrm{CS}}_{n,\alpha}^{\theta; \gamma}   \right\}. 
    \end{split}
    \end{equation*}
    \end{algorithm}
\begin{theorem}
    \label{thm:gamma_ed_method_validity}
Suppose $X_1, \cdots, X_n$ are independent and identically distributed as $F$ (density $f$) which is an absolutely continuous $\gamma$-unimodal distribution on $\mathbb{R}^d$ with a unique mode $\theta_0 \in \mathbb{R}^d$. Then the confidence set returned by \Cref{alg:gamma-conf-int_ed} satisfies the following for all $n\ge1$ and for any $\alpha\in(0, 1)$,
  \[
    \mathbb{P}_F(\theta_0\in\widehat{\mathrm{CS}}_{n,\alpha}^{\gamma}) \ge 1 - \alpha.
    \]
\end{theorem}
\Cref{thm:gamma_ed_method_validity} shows that the proposed confidence set $\widehat{\mathrm{CS}}_{n,\alpha}^{\gamma}$ in \Cref{alg:gamma-conf-int_ed} is finite-sample valid. The proof of this theorem hinges on the fact that the transformed data $\{\|X_i - \theta_0\|_2^{\gamma} \}_{i = 1}^n$ are iid observations from a unimodal univariate distribution with mode $0$. Refer to \Cref{app:gamma_validity} for the detailed proof.

\begin{remark}
In \Cref{alg:gamma-conf-int_ed} we use the transformation $\{\|X_i - \theta\|_2^\gamma\}_{i = 1}^n$ to construct a valid confidence set of the mode. However a valid confidence set of the population mode of a $d$-dimensional $\gamma$-unimodal distribution can also be constructed using the transformed data $\{\pi_S(X_i - \theta)^{\gamma}\}_{i = 1}^n$ where $S \subset \mathbb{R}^d$ is any star-shaped set about $0$ and $\pi_S$ is the corresponding Minkowski functional. 
\end{remark}

\begin{remark}
In general, the optimality of the univariate algorithm $\mathcal{A}(\cdot, \cdot)$ does not imply the minimiax optimal shrinkage to $0$ of the width of the confidence set $\widehat{\mathrm{CS}}_{n,\alpha}^{\gamma}$. 
\end{remark}

\section{Numerical results}
\label{sec:simulation}
In this section we apply the proposed methods in this work to synthetic data and analyze the validity and width of the different confidence sets. We simulate observations from the density $f_{\beta}$, 
\begin{equation*}
                        f_{\beta}(x) = \begin{cases}
                            \frac{-|x|^{\beta}}{2} + \frac{1}{2}, \mbox{ }& - 1 \leq x \leq 0;\\
                            \frac{-\beta^{\beta}x^{\beta}}{2(\beta + 2)^{\beta}} + \frac{1}{2}, \mbox{ } & 0 \leq x \leq \frac{\beta + 2}{\beta};\\
                            0, \mbox{ }&\mbox{otherwise}.
                        \end{cases}
                    \end{equation*}
The density $f_{\beta}$ (for $\beta = 1$) has been studied in Section-$3$ of \cite{sager1975consistency}. The density $f_{\beta}$ satisfies assumption~\ref{assump:f}. Moreover, it can be easily verified that, 
 \[
                    \alpha\left(\delta, 2^{\frac{1}{\beta}}, \frac{2^{\frac{1}{\beta}}(\beta + 2)}{\beta} \right) = \frac{1 - \delta^{\beta}}{1 - 2\delta^{\beta}}   > 1 + \delta^{\beta}\mbox{ for any small } \delta > 0.
                    \]
We check the validity (coverage) and compare the performance (width) of the order statistics based approach (M1), the M-estimation method (M2), the p-value based algorithm which utilizes Edelman's result (M3), and the oracle method (which uses the asymptotic distribution of the estimate of mode). We perform the analysis for different sample sizes $n = 1000, 2000$ and for a range of $\beta$'s. For each sample size and each $\beta$, we perform $1000$ iterations to compute the coverage and the distribution of the width of the confidence sets at level $ 1- \alpha = 0.95$. 
\begin{figure}[h]
                        \centering
                        \includegraphics[width=\textwidth]{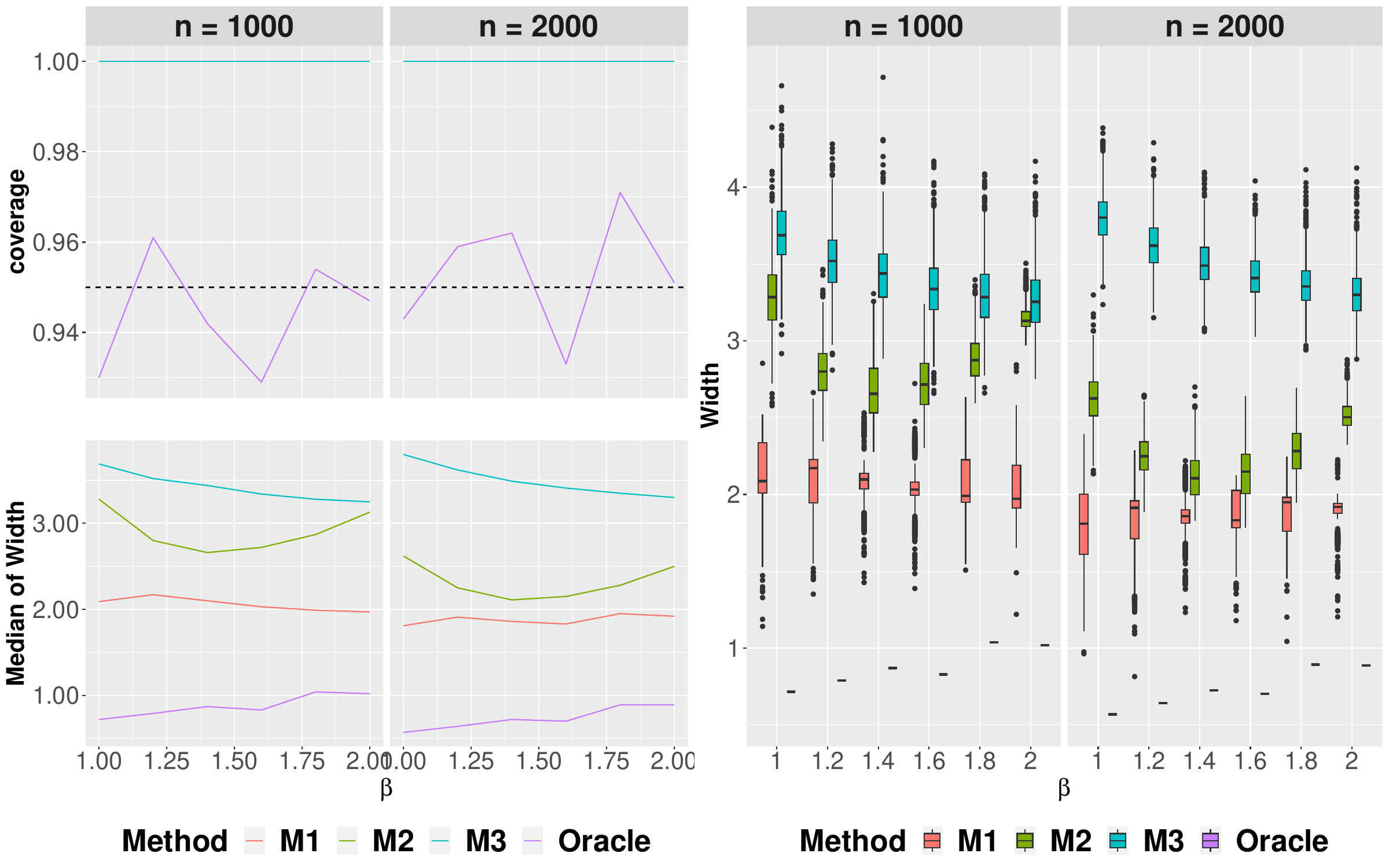}
                        \caption{Comparison of coverage and width of various confidence sets of mode of $f_{\beta}$ for different sample sizes and smoothness parameter, $\beta$.}
                        \label{fig:coverage_guarantee}
                    \end{figure} 
                    
\Cref{fig:coverage_guarantee} demonstrates the performance of the proposed finite sample valid confidence sets of mode of a unimodal distribution. It can be observed from the left panel of the figure that the finite sample valid methods (M1), (M2), (M3) are conservative and always provide the required coverage ($\geq 0.95$) for all sample sizes and all $\beta$ in contrast to the asymptotically valid oracle method. The ``median of width" plot in the left panel and the box plot of width in the right panel reveal that the width of the confidence sets given by (M1) and (M2) gradually decrease as the sample size increases and approach the width of the oracle method. On the contrary, the width of the confidence set given by the p-value based method (M3) doesn't seem to decrease with increasing sample size. All these observations align with the theoretical results discussed in the prior sections. 
\section{Conclusions and Future Directions}
\label{sec:extension}
In this work, we address the problem of constructing finite-sample valid confidence sets for the mode of a unimodal distribution that shrink to a singleton as sample size increases—a task that has received comparatively little attention in the literature relative to point estimation. Building upon classical ideas, we introduce a suite of novel procedures that yield valid confidence sets under minimal assumptions. Our proposed methods leverage spacings between order statistics, M-estimation techniques, and single-observation-based inequalities [\cite{edelman1990confidence}] to achieve finite-sample coverage guarantees. A salient feature of these procedures is their adaptivity to the underlying distribution: under suitable regularity conditions on the density around the mode, we show that the width of the confidence sets contracts at a rate that is minimax optimal (up to logarithmic factors), without requiring prior knowledge of the density’s smoothness.

Furthermore, we extend our framework to the multivariate setting by providing a generic methodology for obtaining finite-sample valid confidence sets for the mode of \(\gamma\)-unimodal distributions in \(\mathbb{R}^d\), using algorithms developed for the univariate case as building blocks. Our theoretical results are complemented by numerical experiments on synthetic datasets. Collectively, our findings contribute new tools for inference on the mode and provide a foundation for future developments in finite-sample uncertainty quantification for statistical functionals.

There are various future directions to this work. Firstly, it is of interest to construct a less conservative finite sample valid confidence interval of the mode of a unimodal univariate distribution with coverage $1 - \alpha$ and whose width shrinks to zero at exactly the minimax optimal rate as the sample size $n \rightarrow \infty$. Secondly, obtaining a finite sample analysis of the width of the confidence set of the mode of a multivariate \(\gamma\)-unimodal distribution [\Cref{alg:gamma-conf-int_ed}] is also of interest. 

\bibliography{references}
\bibliographystyle{plainnat}
\newpage
\appendix
\setcounter{section}{0}
\setcounter{equation}{0}
\setcounter{figure}{0}
\setcounter{remark}{0}
\renewcommand{\thesection}{S.\arabic{section}}
\renewcommand{\theequation}{E.\arabic{equation}}
\renewcommand{\thefigure}{A.\arabic{figure}}
\renewcommand{\theremark}{R.\arabic{remark}}
% \tableofcontents
% \titlelabel{\thetitle: }
% \cftsetindents{section}{1em}{2.5em}
% \cftsetindents{subsection}{1.5em}{3em}
% \setcounter{page}{1}
  \begin{center}
  \Large {\bf Appendix to ``Finite sample valid confidence sets of mode''}
  \end{center}
\section{Finite sample valid confidence set of mode of dependent data}
\label{app:auxi}
We often encounter scenarios where we can not assume the independence of data-points. There are various dependent structures which may occur -- negatively associated (NA), negatively superadditive dependent (NSD), negatively orthant dependent (NOD), extended negatively dependent (END) etc. Refer to \cite{joag1983negative}, \cite{hu2000negatively}, \cite{lehmann2011some}, \cite{liu2009precise} for a more detailed discussion on these dependent structures. It can be shown that all these dependent structures mentioned above fall under the broader category of widely orthant dependent (WOD) data (\cite{wang2013uniform}). A given sequence of random variables $\{X_N\}_{N \geq 1}$ is said to be WOD with dominating coeffients $g(N) = \max\{g_L(N), g_U(N) \}$ if there exists sequences $\{g_L(N)\}_{N\geq 1}, \{g_U(N)\}_{N\geq 1} \geq 1$ such that for any $x_i \in ( -\infty, \infty), $ $ i \leq N$, 
\begin{equation*}
    \begin{split}
        \mathbb{P}(X_1 > x_1, \cdots, X_N > x_N) &\leq g_U(N) \prod_{i = 1}^N \mathbb{P}(X_i > x_i), \\
       \mathbb{P}(X_1 \leq x_1, \cdots, X_N \leq x_N) &\leq g_L(N) \prod_{i = 1}^N \mathbb{P}(X_i \leq x_i) .
    \end{split}
\end{equation*}
These dependent structures often arise in censored and truncated data which are very common in survival analysis, reliability theory, astronomy and economics. \cite{kaber2024kernel} is one of the papers which studies the consistency of the kernel mode estimator for truncated WOD data. In this section we leverage Edelman's result \eqref{eq:edelman} to construct a finite sample valid confidence set of the mode of a unimodal univarite distribution based on data which may have arbitrary dependency between the data-points. The method is described in \Cref{alg:E-conf-int_ed}. 

\begin{algorithm}
    \caption{\textbf{(M3')} Proposed finite sample valid confidence set for the population mode}
    \label{alg:E-conf-int_ed}
  \KwIn{Sample: $X_1,\cdots,X_m$, Confidence Level: $1-\alpha$, tuning parameter $\rho > 1$}
    \KwOut{Confidence set for the population mode with finite sample coverage}
    Divide the sample into two almost equally sized disjoint sub-samples $S_1$ and $S_2$. Let $|S_1| = m - n$ and $|S_2| = n  $. \\
    Compute any consistent estimate $\widehat \theta_1$ of mode based on the sub-sample $S_1$. \\
    Return the confidence set, 
    \[
    \widehat{\mathrm{CS}}_{n, \alpha}^{\mathrm{Ed}_E, \rho} = \left\{\theta \in \mathbb{R}\Bigg| \frac{1}{n}\left(\frac{\rho - 1}{\rho + 1 } \right) \sum_{i \in S_2}\left|\frac{X_i - \theta}{X_i - \widehat \theta_1} \right|^{1/\rho} < \frac{1}{\alpha} \right\}.
    \]
    \end{algorithm}
\begin{theorem}
    \label{thm:E_ed_method_validity}
Suppose $X_1, \cdots, X_m$ are identically distributed as $F$ (density $f$) which is an absolutely continuous unimodal distribution with a unique mode $\theta_0$. Then the confidence set returned by \Cref{alg:proposed-conf-int_ed} satisfies the following for all $n\ge1$, for any $\alpha\in(0, 1)$ and for any $\rho > 1$,
  \[
    \mathbb{P}_F(\theta_0\in \widehat{\mathrm{CS}}_{n, \alpha}^{\mathrm{Ed}_E, \rho}) \ge 1 - \alpha.
    \]
\end{theorem}

\begin{proof}[Proof of \Cref{thm:E_ed_method_validity}]
We use \cite{edelman1990confidence}'s result that for any $t > 1$ and any fixed $ a\in \mathbb{R}$, 
\[
 \mathbb{P}(|X - \theta_0| \leq t|X-a| ) \geq 1 - \frac{2}{t+1}. 
\]
Using this concentration around $\theta_0$ we can obtain a transformation $\psi(X)$ of $X$ which is bounded in $\mathcal{L}_1$. For any fixed $a \in \mathbb{R}$ we have,
\begin{equation*}
    \begin{split}
        \mathbb{E}\left[\left(\frac{\rho - 1}{\rho+1} \right) \left|\frac{X - \theta_0}{X- a} \right|^{1/\rho} \right] &= \left(\frac{\rho - 1}{\rho+1} \right) \mathbb{E}\left[ \left|\frac{X - \theta_0}{X- a} \right|^{1/\rho} \right] \\
        &= \left(\frac{\rho - 1}{\rho+1} \right) \int_0^{\infty} \mathbb{P}\left(\left|\frac{X - \theta_0}{X- a} \right|^{1/\rho} > t  \right) dt \\
        &\leq \left(\frac{\rho - 1}{\rho+1} \right) \left( 1 + \int_1^{\infty} \mathbb{P} \left( \left| \frac{X - \theta_0}{X - a} \right| > t^{\rho} \right) dt  \right) \\
        &\leq \left(\frac{\rho - 1}{\rho+1} \right) \left( 1 + \int_1^{\infty} \frac{2}{t^{\rho} + 1} dt  \right)\\
        &< \left(\frac{\rho - 1}{\rho+1} \right) \left( 1 + \frac{2}{\rho - 1} \right)\\
        &= 1.
    \end{split}
\end{equation*}
We use this $\mathcal{L}_1$ boundedness property to prove the coverage guarantee.  
\begin{equation*}
    \begin{split}
        \mathbb{P} (\theta_0\notin\widehat{\mathrm{CS}}_{n, \alpha}^{\mathrm{Ed}_{E, \rho}} ) &= \mathbb{P} \left( \frac{1}{n}\left(\frac{\rho - 1}{\rho + 1 } \right) \sum_{i \in S_2}\left|\frac{X_i - \theta_0}{X_i - \widehat \theta_1} \right|^{1/\rho} \geq  \frac{1}{\alpha} \right)  \\
        &\stackrel{(i)}{\leq} \alpha \mathbb{E} \left[\frac{1}{n}\left(\frac{\rho - 1}{\rho + 1 } \right) \sum_{i \in S_2}\left|\frac{X_i - \theta_0}{X_i - \widehat \theta_1} \right|^{1/\rho} \right] \\
        &= \alpha \mathbb{E} \left[\left(\frac{\rho - 1}{\rho + 1 } \right) \left|\frac{X - \theta_0}{X - \widehat \theta_1} \right|^{1/\rho} \right] \\
        &\stackrel{(ii)}{\leq} \alpha. 
    \end{split}
\end{equation*}
The step-$(i)$ follows because of Markov's inequality. The step-$(ii)$ holds because of the $\mathcal{L}_1$ boundedness property of the transformed data (with $a = \widehat \theta_1$). This completes the proof of the theorem. 
\end{proof}
In addition to being identically distributed, if the data $\{X_i\}_{i = 1}^m$ are also independent, it can be shown that the confidence set converges almost surely to a limiting set as the sample size $m$ goes to infinity. 
\begin{theorem}
\label{thm:E_ed_width}
Suppose $X_1, \cdots, X_m$ are independent and identically distributed as $F$ (density $f$) which is an absolutely continuous unimodal distribution with a unique mode $\theta_0$ such that $\mathbb{E}_F[ |(X-\theta)/(X - \theta_0)|^{1/\rho}] $ exists for all $\theta \in \mathbb{R}$. Let $\widehat \theta_1$ be a consistent estimator of $\theta_0$. Then we have, 
\[
\widehat{\mathrm{CS}}_{n, \alpha}^{\mathrm{Ed}_{E, \rho}} \stackrel{a.s.}{\rightarrow} \left\{\theta \in \mathbb{R} \Bigg| \mathbb{E}_F \left[\left(\frac{\rho - 1}{\rho + 1 } \right) \left|\frac{X - \theta}{X -  \theta_0} \right|^{1/\rho}\right] < \frac{1}{\alpha}\right\} \quad \mbox{as} \quad m-n,n \rightarrow \infty,
\]
\end{theorem}
\begin{proof}[Proof of \Cref{thm:E_ed_width}]
The proof is a simple application of the law of large numbers. We have the following equivalence relation, 
\begin{equation*}
        \begin{split}
            \theta \in  \widehat{\mathrm{CS}}_{n, \alpha}^{\mathrm{Ed}_{E, \rho}}& \iff  \left(\frac{\rho - 1}{\rho + 1 } \right) \frac{1}{n} \sum_{i \in S_2}\left|\frac{X_i - \theta}{X_i - \widehat \theta_1} \right|^{1/\rho} < \frac{1}{\alpha} . 
        \end{split}
    \end{equation*}
As $n \rightarrow \infty$ the term in the left hand side converges almost surely to $((\rho - 1)/(\rho + 1))\mathbb{E}[|(X - \theta)/(X - \theta_0)|^{1/\rho}]$ by strong law of large numbers (we also use the fact that $\widehat \theta_1$ is a consistent estimator of $\theta_0$). Thus as $m-n, n \rightarrow \infty$, 
\[
\widehat{\mathrm{CS}}_{n, \alpha}^{\mathrm{Ed}_{E, \rho}} \stackrel{a.s.}{\rightarrow} \left\{\theta \in \mathbb{R} \Bigg| \mathbb{E}_F \left[\left(\frac{\rho - 1}{\rho + 1 } \right) \left|\frac{X - \theta}{X - \theta_0} \right|^{1/\rho}\right] < \frac{1}{\alpha}\right\}
\]
This completes the proof of the theorem. 
\end{proof}

\section{Proof of \Cref{thm:main_result_os_validity}}
\label{app:thm_os_valid}
We begin the proof by considering the following set of distribution functions, 
\[
\mathcal{C}_n(\alpha/2) = \left\{ F: c_B \leq F(I_{Bi}) \leq d_B \quad \mbox{for } i=1,\cdots,n_B \quad \mbox{and for } B = 0,\cdots,B_{max}\right\},
\]
where,
\begin{equation*}
    \begin{split}
        c_B &= q\mbox{Beta}(\alpha/(4(B+2)n_Bt_n),2^{B+s_n},n+1 - 2^{B+s_n}), \\
        d_B &= q\mbox{Beta}(1 -(\alpha/(4(B+2)n_Bt_n)),2^{B+s_n},n+1 - 2^{B+s_n}).
    \end{split}
\end{equation*}
We know from Section $2.1$ of \cite{walther2022confidence} that,
\begin{equation}
    \label{eq:ci_band}
   \mathbb{P}_F(F \in \mathcal{C}_n (\alpha)) \geq 1-(\alpha/2), 
\end{equation} 
for any distribution function $F$. Our proof is based on the following fundamental property of unimodal distribution function, for two adjacent disjoint intervals $I_1 = [a,b],I_2 = [b,c]$ (where $-\infty<a<b<c<\infty$) if $F(I_1)/|I_1| > F(I_2)/|I_2|$ then either the mode $\theta_0 \in I_1$ or $\theta_0 <a$ i.e.\ the mode is "closer" to the interval $I_1$ as compared to $I_2$. Thus we can drop $I_2$ and the intervals present after $I_2$ from the collection of intervals that can contain the mode. Similarly if $F(I_1)/|I_1| < F(I_2)/|I_2|$ then either the mode $\theta_0 \in I_2$ or $\theta_0 > c$ i.e.\ the mode is "closer" to the interval $I_2$ as compared to $I_1$. Thus we can drop $I_1$ and the intervals present before $I_1$ from the collection of intervals that can contain the mode.

We note from the construction that for $0\leq B \leq B_{\max}$ the intervals $\{I_{Bi}\}_{i=1}^{n_B}$ are disjoint and every pair of consecutive intervals in this set share a common boundary. Moreover the intervals $\{I_{0i}\}_{i=1}^{n_0}$ are subsets of the intervals $\{I_{1i}\}_{i=1}^{n_1}$ which are in fact subsets of the intervals $\{I_{2i}\}_{i=1}^{n_2}$ and so on till $\{I_{B_{max}i}\}_{i=1}^{n_{B_{max}}}$. Using the fact that $\mathbb{P}_F(F \in \mathcal{C}_n (\alpha)) \geq 1-(\alpha/2)$ we note that with probability greater than or equal to $1-(\alpha/2)$ the following holds for all $1 \leq i \leq n_B$ and for all $0 \leq B \leq B_{max}$,
\begin{equation}
    \label{eq:ratio_bounds}
    \frac{c_B}{|I_{Bi}|} \leq \frac{F(I_{Bi})}{|I_{Bi}|} \leq \frac{d_B}{|I_{Bi}|}.
\end{equation}
We start with $B=B_{max}$ and find the interval $I_{B_{max}i_{min}}$. By \eqref{eq:ratio_bounds} $I_{B_{max}i_{min}}$ has the highest lower bound in the set $\{I_{B_{max}i}\}_{i=1}^{n_{B_{max}}}$. We now consider the set $K_{B_{max}}$. We observe that for all intervals $I_{B_{max}i} \in K_{B_{max}}$ the following holds,
\begin{equation*}
    \begin{split}
        &\mbox{Width}(I_{B_{max}i}) \leq h_{B_{max}} \mbox{Width}(I_{B_{max}i_{min}}) \\
        \implies & |I_{B_{max}i}| \leq \frac{d_{B_{max}}}{c_{B_{max}}} |I_{B_{max}i_{min}}| \\
        \implies &\frac{c_{B_{max}}}{|I_{B_{max}i_{min}}|} \leq \frac{d_{B_{max}}}{|I_{B_{max}i}|}.
    \end{split}
\end{equation*}
Thus $K_{B_{max}}$ comprises of the largest set of neighbouring intervals of $I_{B_{max}i_{min}}$ which have overlapping bounds with that of $I_{B_{max}i_{min}}$ i.e.\ $K_{B_{max}}$ consists of those intervals $I_{B_{max}i}$ for which the upper bound of $F(I_{B_{max}i})/|I_{B_{max}i}|$ is greater than the lower bound of $F(I_{B_{max}i_{min}})/|I_{B_{max}i_{min}}|$. From the fundamental property of unimodal distribution functions discussed before, we drop all the intervals in $\{I_{B_{max}i}\}_{i=1}^{n_{B_{max}}}$ which are not in the set $K_{B_{max}}$. Continuing like this for $B=B_{max}-1, \cdots, 0$ we finally obtain $K_0$. By the same argument we know that the mode $\theta_0$ can not lie in any interval in the set $\cup_{i=1}^{n_{B_0}}\{I_{B_0i}\} \backslash K_0$. Thus either $\theta_0 \in \left(-\infty, X_{(1)} \right]$ or $\theta_0 \in K_0$ or $\theta_0 \in \left[ X_{(n)}, \infty \right)$. However we know from the fundamental property of unimodal distribution functions that $\theta_0$ can lie in $\left(-\infty, X_{(1)} \right]$ iff the left most interval $I_{01} \in K_0$. Similarly we know that $\theta_0$ can lie in $\left[ X_{(n)}, \infty \right)$ iff the right most interval $I_{0n_0} \in K_0$. 

Here we use Theorem-$4$ of \cite{lanke1974interval} which states that,
\begin{equation}
\label{eq:lanke_result}
\mathbb{P} (\theta_0 \in (X_{(1)} - \lambda_{n,\alpha}(X_{(n)}-X_{(1)}), X_{(n)} + \lambda_{n,\alpha}(X_{(n)}-X_{(1)})) ) \geq 1-(\alpha/2),
\end{equation}
where $\lambda_{n,\alpha}$ is as described in \Cref{alg:proposed-conf-int_os}. Combining the previous discussion and the high probability statements in \eqref{eq:ci_band} and \eqref{eq:lanke_result} we conclude that,
\[
\mathbb{P}_F(\theta_0\in\widehat{\mathrm{CI}}_{n,\alpha}) \geq (1-(\alpha/2))+(1-(\alpha/2)) -1 \geq 1-\alpha. 
\]
This completes the proof of the theorem.

\section{Proof of \Cref{thm:width_analysis_os}}
\label{app:thm_os_width}
From \Cref{alg:proposed-conf-int_os} we can see that $h_B$ is a decreasing function of $\alpha$. Thus because of the nested nature of the algorithm we have $\widehat{\mathrm{CI}}_{n,\alpha} \subset \widehat{\mathrm{CI}}_{n,1/n}$ and hence $\mathrm{Width}(\widehat{\mathrm{CI}}_{n,\alpha}) \leq \mathrm{Width}(\widehat{\mathrm{CI}}_{n,1/n})$. Because of this it is enough to prove the theorem for $\alpha = 1/n$. Therefore in the remaining proof we assume $\alpha = 1/n$ without loss of generality. We note that the number of observations contained in $I_{Bi}$ is $2^{B+s_n} + 1$ which varies from $\log(n) + 1$ to $(n/8) + 1$ as $B$ ranges from $0$ to $B_{max}$. Let $\log(n) < r(n) = An^{2\beta/(1+2\beta)} < n/8$ for some $A>0$ and let $0 \leq B(r) = \log_2(r(n)) - s_n  \leq B_{\max}$. We note that the number of observations contained in $I_{B(r)i}$ is $r(n) + 1$. We re-label the interval formed by the set $K_{B(r) +1}$ as $[c,d]$.  If there are at-least $r(n) + 1$ observations in $[c,d]$ we let $K(n)$ be the discrete random variable such that the following holds. $[ X_{K(n)}, X_{K(n) + r(n)} ] = I_{B(r)i}$ for some $i$, $[ X_{K(n)}, X_{K(n) + r(n)} ] \subset [c,d]$ and,
\[
X_{K(n) + r(n)} - X_{K(n)} = \min\{\mbox{Width}(I_{B(r)i})|\mbox{ } I_{B(r)i} \subset [c,d] \}.
\]
If $[c,d]$ does not contain $r(n) + 1$ observations then $K(n)$ can be set to any arbitrary number (does not affect the proof). It is important to note that by definition $[ X_{K(n)}, X_{K(n) + r(n)} ] = I_{B(r)i_{min}}$. Since $\widehat{\mathrm{CI}}_{n,\alpha}$ contains $\theta_0$ with probability at-least $1-(1/n)$ we have that $\theta_0 \in [c,d]$ with probability at-least $1-(1/n)$. For showing the in probability convergence we can therefore do the remaining computations under the event $\{\theta_0 \in [c,d] \}$. We also let $J(n)$ be the following discrete random variable. If $[\theta_0, d]$ contains at-least $r(n) + 1$ observations then we define $J(n)$ to be the smallest index $j$ such that $[ X_{j}, X_{j + r(n)} ] = I_{B(r)i}$ for some $i$, $[ X_{j}, X_{j + r(n)} ] \subset [\theta_0,d]$. If $[\theta_0,d]$ does not contain $r(n) + 1$ observations then $J(n)$ can be set to any arbitrary number. We need the following lemma from \cite{sager1975consistency}. 
\begin{lemma}
    \label{lem:sager}
    Let $S_1,S_2,\cdots$ and $T_1,T_2,\cdots$ be a sequence of random variables such that $S_n \leq T_n$ for each $n$ and $[S_n, T_n]$ contains $r(n) + 1$ observations where $r(n)$ is of the form $An^v$, $0 < v <1$. Then we have the following, 
    \[
    \frac{F(T_n) - F(S_n)}{F_n(T_n) - F_n(S_n)} = 1 + o(r(n)^{-1/2}log(r(n))) \quad \mbox{with probability one}. 
    \]
\end{lemma}
We break the proof of the theorem into three cases. At first we prove the result for the case when $\theta_0 = c$, then we prove the result for the case when $\theta_0 = d$, and finally we prove for the case when $\theta_0 \in (c, d)$. 

Suppose $\theta_0 = c$. We show that $X_{J(n) + r(n)} = c + o(\delta(n))$ with probability $1$ where $\delta(n) = n^{-1/(1+ 2\beta)}(\log n )^{1/\beta}$. We observe that $r^+(R_2\delta) \rightarrow f(c)$ as $\delta \rightarrow 0$. If this does not happen then we can say from the definition of $r^+(R_2(\delta)$ that for some $\epsilon >0$ and for all $c <x < d$ we have $f(x) + \epsilon < f(c)$. This implies that $f(c) + \epsilon < f(c)$, a contradiction to the assumption of \textit{standard conditions}. 

Combining the above with the fact that $r(\delta, R_2\delta) \geq 1 + \rho \delta^{\beta}$ we can say that there exists a constant $B > 0$ such that $B < r^+(R_2\delta) < r^-(\delta) \leq r^-(\epsilon \delta)$ for all small $\delta$ and for each $0 < \epsilon < 1$. Thus for $0 < \epsilon <1$ and for all large $n$ the following holds true, 
\[
F(c + \epsilon \delta_n) - F(c) \geq \int_{c}^{c + \epsilon \delta_n}r^-(\epsilon\delta(n)) dx > B \epsilon \delta(n). 
\]
This implies that, 
\begin{equation}
    \label{eq:epsilondelta}
  \liminf_{n \rightarrow \infty} \frac{F(c + \epsilon \delta(n)) - F(c)}{B \epsilon \delta(n)} \geq 1 .   
\end{equation}
On the other hane using \Cref{lem:sager} and the definition of $J(n)$ we get that, 
\[
1 \leq \liminf_{n \rightarrow \infty}\frac{F(X_{J(n) + r(n)}) - F(c)}{r(n)/n} \leq \limsup_{n \rightarrow \infty} \frac{F(X_{J(n) + r(n)}) - F(c)}{r(n)/n} \leq 2 \mbox{ with probability }1.
\]
We note that $n^{-1}r(n)/(B\epsilon \delta(n)) \rightarrow 0$ as $n \rightarrow \infty$. Using this in the previous equation we obtain that, 
\begin{equation}
    \label{eq:J_n}
    \lim_{n \rightarrow \infty} \frac{F(X_{J(n) + r(n)}) - F(c)}{B\epsilon \delta(n)} \rightarrow 0 \mbox{ with probability }1.
\end{equation}
The equations \eqref{eq:epsilondelta} and \eqref{eq:J_n} imply that $F(c + \epsilon \delta(n)) > F(X_{J(n) + r(n)})$ with probability $1$. Since $\epsilon$ can be chosen arbitrarily we conclude that $X_{J(n) + r(n)} = c + o(\delta(n))$ with probability $1$. 

Suppose $L(n)$ is the discrete random variable such that $[X_{L(n)} ,X_{L(n) + r(n)}] = I_{B(r)i_{ex}} \subset K_{B(r)}$ where $I_{B(r)i_{ex}}$ is the right extreme interval present in the set $K_{B(r)}$. Analysis can be similarly done for the left extreme interval present in the set $K_{B(r)}$. We show that $X_{L(n)} = c + o(\delta(n))$ with probability $1$. From \Cref{alg:proposed-conf-int_os} we know that, 
\[
\frac{X_{L(n) + r(n)} - X_{L(n)}}{X_{K(n) + r(n)} - X_{K(n)}} \leq h_{B(r)} = \frac{d_{B(r)}}{c_{B(r)}}. 
\]
Here we use a result from \cite{hengartner1995finite}. Suppose $\tau$ is a constant such that, 
\[
2\beta + \frac{2 \tau^3}{3} \sqrt{\frac{\log(n)}{(1 + 2\beta)n^{2\beta/(1 + 2\beta)}}} \leq \tau^2 - 2.
\]
Then the following holds true, 
\begin{equation*}
    \begin{split}
        \lim_{n \rightarrow \infty} \left(c_{B(r(n))} - \frac{r(n) - \sqrt{r(n)} \tau \sqrt{\log(n/r(n))}}{n + \sqrt{n}z_{1 - \alpha/2}}\right) &\geq 0 \\
        \lim_{n \rightarrow \infty} \left(d_{B(r(n))} - \frac{r(n) +\sqrt{r(n)} \tau \sqrt{\log(n/r(n))}}{n - \sqrt{n}z_{1 - \alpha/2}}\right) &\leq 0 .
    \end{split}
\end{equation*}
We now analyze the behavior of $h_{B(r)}$. All the inequalities in the following equation are true upto a constant. 
\begin{equation*}
    \begin{split}
        h_{B(r(n))} &= \frac{d_{B(r(n))}}{c_{B(r(n))}} \\
        &\leq \frac{r(n) +\sqrt{r(n)} \tau \sqrt{\log(n/r(n))}}{n - \sqrt{n}z_{1 - \alpha/2}}\frac{n + \sqrt{n}z_{1 - \alpha/2}}{r(n) - \sqrt{r(n)} \tau \sqrt{\log(n/r(n))}} \\
        &= \frac{n^{2\beta/(1+2\beta)} + n^{\beta/(1+2\beta)\sqrt{\log(n^{1/(1+2\beta)})}}}{n^{2\beta/(1+2\beta)} - n^{\beta/(1+2\beta)\sqrt{\log(n^{1/(1+2\beta)})}}}\frac{n + \sqrt{n}z_{1 - \alpha/2}}{n - \sqrt{n}z_{1 - \alpha/2}} \\
        &= \left(1 + \frac{1}{n^{\beta/(1+2\beta)}(\log n)^{-1/2}} \right)\left(1 + o(n^{-1/2})\right) \\
        &= (1 + o(n^{-\beta/(1 +2\beta)}\log(n)))(1 + o(n^{-1/2})) \\
        &= 1 + o(n^{-\beta/(1 +2\beta)}\log(n)).
    \end{split}
\end{equation*}
Suppose $\Omega_L$ is the set where our claim that $X_{L(n)} = c + o(\delta(n))$ with probability $1$ does not hold true. If $\omega \in \Omega_L$, there exists a subsequence $\{n(j)\}$ and and $\epsilon$ such that $X_{L(n(j))} > c + R_2 \epsilon \delta(n(j))$ for all $j$. 

Setting $S_n = X_{L(n)}, T_n = X_{L(n) + r(n)}$ and $S_n = X_{J(n)}, T_n = X_{J(n) + r(n)}$ and using \Cref{lem:sager} we get that the following holds with probability $1$, 
\begin{equation*}
    \begin{split}
        \frac{F(X_{L(n) + r(n)}) - F( X_{L(n)})}{r(n)/n} &= 1 + o(r(n)^{-1/2}\log(r(n))), \\
       \frac{F(X_{J(n) + r(n)}) - F( X_{J(n)})}{r(n)/n} &= 1 + o(r(n)^{-1/2}\log(r(n))).
    \end{split}
\end{equation*}
Using the assumption made in \Cref{thm:width_analysis_os} we have, 
\begin{equation*}
    \begin{split}
        1 + \rho(\epsilon \delta(n(j)))^{\beta} &\leq r(\epsilon \delta(n(j)), R_2\epsilon\delta(n(j))) \\
        &\leq \frac{r^-(\epsilon \delta(n(j)))}{r^+(R_2\epsilon \delta(n(j)))}\frac{X_{J(n(j)) + r(n(j))} - X_{J(n(j))}}{X_{K(n(j)) + r(n(j))} - X_{K(n(j))}} \\
        &\leq \frac{r^-(\epsilon \delta(n(j)))}{r^+(R_2\epsilon \delta(n(j)))} \frac{X_{J(n(j)) + r(n(j))} - X_{J(n(j))}}{X_{L(n(j)) + r(n(j))} - X_{L(n(j))}} h_{B(r(n(j)))} \\
        &\leq \frac{(F(X_{J(n(j)) + r(n(j))}) - F( X_{J(n(j))}))}{(F(X_{L(n(j)) + r(n(j))}) - F( X_{L(n(j))}))}\frac{n(j)r(n(j))^{-1}}{n(j)r(n(j))^{-1}} h_{B(r(n(j)))} \\
        &\leq (1 + o(r(n(j))^{-1/2}\log(r(n(j)))))(1 + o(n(j)^{-\beta/(1 +2\beta)}\log(n(j)))) \\
        &= (1 + o(n(j)^{-\beta/(1 +2\beta)}\log(n(j))))(1 + o(n(j)^{-k/(1 +2k)}\log(n(j)))) \\
        &= (1 + o(n(j)^{-\beta/(1 +2\beta)}\log(n(j)))).
    \end{split}
\end{equation*}
This is a contradiction because $ 1 + \rho(\epsilon \delta(n(j)))^{\beta}  = 1 + \rho \epsilon^{\beta}n(j)^{-\beta/(1+2\beta)}\log(n(j))$ and this violates the above inequality for large $j$. Thus we conclude that $X_{L(n)} = c + o(\delta(n))$ with probability $1$. Moreover with probability $1$ we also have, 
\begin{equation*}
    \begin{split}
     X_{L(n) + r(n)} - c &= (X_{L(n) + r(n)} - X_{L(n)}) + X_{L(n)} - c \\
     &\leq h_{B(r(n))} (X_{K(n) + r(n)} - X_{K(n)})  + X_{L(n)} - c \\
     &\leq h_{B(r(n))} (X_{J(n) + r(n)} - X_{J(n)}) + X_{L(n)} - c \\
     &\leq (1 + o(\delta(n)^{\beta}))o(\delta(n)) + o(\delta(n)) \\
     &= o(\delta(n)). 
    \end{split}
\end{equation*}
In the above we use the fact that $h_{B(r(n))} \leq 1 + o(n^{-\beta/(1 +2\beta)}\log(n)) = 1 + o(\delta(n)^{\beta})$, $(X_{J(n) + r(n)} - X_{J(n)}) \leq (X_{J(n) + r(n)} - c) = o(\delta(n))$, and $X_{L(n)} - c = o(\delta(n))$. We note that $\widehat{\mathrm{CI}}_{n,\alpha} \subset K_{B(r)}$ and $X_{L(n) + r(n)}$ is the right end-point of the confidence interval formed by $K_{B(r)}$. A similar result can be obtained for the left end-point of the confidence interval formed by $K_{B(r)}$. Since both the end-points behave as $c + o(\delta(n))$, the width of the confidence interval formed by $K_{B(r)}$ and thus also $\mbox{Width}(\widehat{\mathrm{CI}}_{n,\alpha})$ behaves as $o(\delta(n))$. This completes the proof for the case when $\theta_0 = c$. The proof for the case when $\theta_0 = d$ is similar. 

Now suppose $\theta_0 \in (c, d)$. We define discrete random variables $K_1(n)$ and $K_2(n)$ as follows. $[X_{K_1(n)}, X_{K_1(n) + r(n)}] = I_{B(r)i}$ for some $i$, $[X_{K_1(n)}, X_{K_1(n) + r(n)}] \subset [c, \theta_0]$ and, 
\[
X_{K_1(n) + r(n)} - X_{K_1(n)} = \min\{\mbox{Width}(I_{B(r)i}) | I_{B(r)i} \subset [c, \theta_0] \}.
\]
Similarly $[X_{K_2(n)}, X_{K_2(n) + r(n)}] = I_{B(r)i}$ for some $i$, $[X_{K_2(n)}, X_{K_2(n) + r(n)}] \subset [\theta_0, d]$ and, 
\[
X_{K_2(n) + r(n)} - X_{K_2(n)} = \min\{\mbox{Width}(I_{B(r)i}) | I_{B(r)i} \subset [\theta_0, d] \}.
\]
Suppose $L_1(n)$ and $L_2(n)$ are discrete random variables such that $[X_{L_1(n)}, X_{L_1(n) + r(n)}] = I_{B(r)i} \mbox{ (for some $i$) } \subset [c, \theta_0]$, $[X_{L_2(n)}, X_{L_2(n) + r(n)}] = I_{B(r)i} \mbox{ (for some $i$) } \subset [\theta_0, d]$ and, 
\begin{equation*}
    \begin{split}
        \frac{X_{L_1(n) + r(n)} - X_{L_1(n)}}{X_{K_1(n) + r(n)} - X_{K_1(n)}} &\leq h_{B(r(n))} , \\
        \frac{X_{L_2(n) + r(n)} - X_{L_2(n)}}{X_{K_2(n) + r(n)} - X_{K_2(n)}} &\leq h_{B(r(n))}.
    \end{split}
\end{equation*}

From the previous cases we know that $X_{L_1(n)} = c + o(\delta(n))$ and $X_{L_2(n) + r(n)} = c + o(\delta(n))$ with probability $1$. We observe that $\widehat{\mathrm{CI}}_{n,\alpha} \subset \mbox{ C.I. generated by }K_{B(r(n))} \subset [X_{L_1(n)}, X_{L_2(n) + r(n)}]$. Thus we have $\mbox{Width}(\widehat{\mathrm{CI}}_{n,\alpha}) \leq X_{L_2(n) + r(n)} - X_{L_1(n)} = o(\delta(n))$ with probability $1$. In other words $n^{1/(1+2\beta)}(\log n)^{-1/\beta}\mbox{Width}(\widehat{\mathrm{CI}}_{n,\alpha}) \stackrel{P}{\rightarrow} 0$ as $n \rightarrow \infty$. This completes the proof of the theorem. 

\section{Proof of \Cref{thm:m_estimation}}
\label{app:m_validity}
We note that $\theta_0 = \mbox{arg max}_x(f(x))$. Let $g_h$ be the density of a Uniform$[-h,h]$ random variable. We consider the convolution $f*g_h$, 
\begin{equation*}
    \begin{split}
        f*g_h(y) &= \int_{-h}^h f(y - x)g_h(x) dx \\
        &= \frac{1}{2h} \int_{-h}^h  f(y - x) dx \\
        &= \frac{F(y + h) - F(y - h)}{2h} \\
        &= \mathbb{E} \left[\frac{1}{2h} \textbf{1}\{ y - h < X \leq y + h \}\right]. 
    \end{split}
\end{equation*}
Let $F_h$ be the corresponding distribution function of the convolution $f*g_h$. Let $\theta_h = \mbox{arg max}_y(f*g_h(y))$ be the population mode of the convolution $f*g_h$. We note the following characterization of $\theta_h$, 
\begin{equation*}
    \begin{split}
\theta_h &=  \mbox{arg min}_y( - f*g_h(y)) \\
&= \mbox{arg min}_y \mathbb{E} \left[-\frac{1}{2h} \textbf{1}\{ y - h < X \leq y + h \}\right] \\
& = \mbox{arg min}_y Pm_{y;h}, 
    \end{split}
\end{equation*}
where $m_{y;h}$ is as defined in \Cref{alg:M_proposed-conf-int}. We further observe the following for $y \geq \theta_0 + h$, 
\begin{equation*}
    \begin{split}
        f*g_h(y) &= \int_{-h}^h f(y - x)g_h(x) dx \\
        &\leq \int_{-h}^h f(\theta_0 + h - x)g_h(x) dx \\
        &= f*g_h(\theta_0 + h). 
    \end{split}
\end{equation*}
We can similarly show that $f*g_h(y) \leq f*g_h(\theta_0 - h)$ if $y \leq \theta_0 - h$. Thus the mode of the convolution, $\theta_h = \mbox{arg max}_y(f*g_h(y)) \in (\theta_0 - h, \theta_0 + h)$. This implies that, 
\begin{equation}
    \label{eq:theta_0_inclusion}
    \theta_0 \in (\theta_h - h, \theta_h + h). 
\end{equation}
This suggests that if we can compute a finite sample valid $(1 - \alpha)$ confidence set $\tilde{\mathrm{CI}}_{n, \alpha}^h$ of $\theta_h$, then $\tilde{\mathrm{CI}}_{n, \alpha}^h \pm h$ will be a finite sample valid $(1 - \alpha)$ confidence set of $\theta_0$. From Theorem-$2$ in \cite{takatsu2025bridging} we see that if we can find a random mapping $\theta \rightarrow L_{n,\alpha}(\theta, \widehat \theta_1; S_2) = P_n(m_{\theta;h} - m_{\widehat \theta_1; h}) - t(\alpha, \theta, \widehat \theta_1)$ (for some non-negative function $t(\cdot, \cdot, \cdot)$) such that, 
\begin{equation}
\label{eq:m_condition}
    \mathbb{P}(P(m_{\theta_h;h} - m_{\widehat \theta_1;h}) \geq L_{n,\alpha}(\theta, \widehat \theta_1; S_2) |S_1) = \mathbb{P}((P_n - P)(m_{\theta_h;h} - m_{\widehat \theta_1;h}) \leq t(\alpha, \theta_h, \widehat \theta_1) |S_1) \geq 1 - \alpha,
\end{equation}
then $\tilde{\mathrm{CI}}_{n, \alpha}^h = \{\theta \in \mathbb{R}| L_{n,\alpha}(\theta, \widehat \theta_1; S_2) \leq 0 \}$ is a finite sample valid $(1 - \alpha)$ confidence set of $\theta_h$. 

We note that $(m_{\theta_h;h} - m_{\widehat \theta_1;h})(X)$ is a trinomial distribution that takes value $-1/2h, 0, 1/2h$. We use the following one-sided concentration inequality for the mean of trinomial distribution. 
\begin{lemma}
\label{lem:trinomial}
 Let $Z_1, \cdots, Z_n$ be iid random variables from a trinomial distribution which takes values $t_1, t_2, t_3$ with probabilities $p, q, 1- p - q$ respectively. Let $\mu$ and $\widehat \mu$ be the population mean and sample mean respectively. Then for $\alpha \in (0,1)$ we have, 
 \[
 \mathbb{P}\left(\widehat \mu - \mu \leq (|t_1|+|t_2|+|t_3|)\sqrt{\frac{3}{2n}}\left[\sqrt{\log(1/\alpha)} + 2 \right] \right) \geq 1 - \alpha. 
 \]
\end{lemma}
Using \Cref{lem:trinomial} with $(t_1, t_2, t_3 ) = (-1/2h, 0, 1/2h)$ we get that $t(\alpha, \theta_h, \widehat \theta_1) = \sqrt{3/(2n)}(1/h)[\sqrt{\log(1/\alpha)} + 2]$ satisfies the condition \eqref{eq:m_condition}. Hence, 
\begin{equation*}
\begin{split}
 \widehat{\mathrm{CS}}_{n, \alpha}^h & = \tilde{\mathrm{CS}}_{n, \alpha}^h \pm h \\
 &= \{\theta \in \mathbb{R}| L_{n,\alpha}(\theta, \widehat \theta_1; S_2) \leq 0 \} \pm h\\
 &= \{\theta \in \mathbb{R}| P_n(m_{\theta;h} - m_{\widehat \theta_1; h})  \leq t(\alpha, \theta, \widehat \theta_1) \} \pm h \\
 &= \left\{\theta \in \mathbb{R} \Bigg| P_{n}(m_{\theta; h} - m_{\widehat{\theta}_1;h}) \leq \frac{1}{h}\sqrt{\frac{3}{2n}}\left[\sqrt{\log(1/\alpha)} + 2 \right] \right\} \pm h 
\end{split}
\end{equation*}
is a finite sample $(1 - \alpha)$ confidence set of $\theta_0$. This complete the proof of the theorem.

We now provide a proof of \Cref{lem:trinomial}.
\begin{proof}[Proof of \Cref{lem:trinomial}]
Let $p_n = P_nI(Z_i = t_1), q_n = P_nI(Z_i = t_2)$ be the sample estimates of $p, q$ respectively. Let $Q, Q_n$ be the trinomial distributions which take values $t_1, t_2, t_3$ with probabilities $p,q, 1-p-q$ and $p_n, q_n, 1 - p_n - q_n$ respectively. For any $\epsilon > 0$ we have the following, 
\begin{equation*}
    \begin{split}
      &  \mathbb{P}\left( \widehat \mu - \mu \leq (|t_1| + |t_2| + |t_3|)\epsilon \right) \\
     \geq & \mathbb{P}\left( |\widehat \mu - \mu| \leq (|t_1| + |t_2| + |t_3|)\epsilon \right) \\
    =& \mathbb{P}\left( |(p_n - p)t_1 + (q_n - q)t_2 + (p+ q - p_n - q_n)t_3 |\leq (|t_1| + |t_2| + |t_3|)\epsilon \right) \\
    \geq& \mathbb{P}(\delta(Q_n, Q) \leq \epsilon) \\
    \stackrel{(i)}{\geq}&  \mathbb{P}(\sqrt{(1/2)KL(Q_n||Q)} \leq \epsilon) \\
    = & \mathbb{P}(KL(Q_n||Q) \leq 2\epsilon^2) \\
    \stackrel{(ii)}{\geq}& 1- \exp(-2n\epsilon^2)\left(e\epsilon^2n \right)^2. 
    \end{split}
\end{equation*}
Here $\delta(Q_n, Q) = \sup\{|Q_n(A) - Q(A)|: A \mbox{ is measurable} \}$ is the total variation distance between $Q_n$ and $Q$ and $KL(Q_n, Q)$ is the Kullback-Leibler divergence between $Q_n$ and $Q$. In step-$(i)$ we use the Pinsker's inequality [\cite{reid2009generalised}] that $\delta(Q_n, Q) \leq \sqrt{(1/2) KL(Q_n, Q)}$. In step-(ii) we use Theorem-$1.2$ of \cite{agrawal2020finite} with $k = 3$. We now observe that for $\epsilon^2 = (3/(2n))\log(9e^2/(2\alpha))$ we have, 
\begin{equation*}
    \begin{split}
  \exp(-2n\epsilon^2)\left(e\epsilon^2n \right)^2 &= \left(\frac{2\alpha}{9e^2}\right)^3  \left(\frac{3e}{2} \right)^2\log^2\left( 9e^2/(2\alpha) \right)  \\
  &\stackrel{(i)}{\leq}  \left(\frac{3e}{2} \right)^2  \left(\frac{2\alpha}{9e^2}\right)^3  \left(1 - 2\frac{9e^2}{2\alpha} + \left(\frac{9e^2}{2\alpha} \right)^2 \right) \\
  &\leq \left(\frac{3e}{2} \right)^2  \frac{4\alpha}{9e^2} \\
  & \leq \alpha. 
    \end{split}
\end{equation*}
In step-$(i)$ we use $\log(1+x) <x $ for $x > 0$. Therefore we see that with probability at-least $ 1 - \alpha$, 
\begin{equation*}
    \begin{split}
 \widehat \mu - \mu & \leq (|t_1| + |t_2| + |t_3|)\sqrt{\frac{3}{2n}}\sqrt{\log\left(\frac{9e^2}{2 \alpha} \right)} \\
 &\leq (|t_1| + |t_2| + |t_3|)\sqrt{\frac{3}{2n}}\left[\sqrt{\log(1/\alpha)} + \sqrt{\log\left(\frac{9e^2}{2 } \right)} \right] \\
&\leq (|t_1| + |t_2| + |t_3|)\sqrt{\frac{3}{2n}}\left[\sqrt{\log(1/\alpha)} + 2 \right].
    \end{split} 
\end{equation*}
This completes the proof of the lemma. 
\end{proof}

\section{Proof of \Cref{thm:m_estimator_width}}
\label{app:proof_m_width}
In this proof we redefine $m_{\theta;h}(X) = -\textbf{1}\{\theta - h < X \leq \theta+ h\}$. Thus the confidence set $\widehat{\mathrm{CS}}_{n,\alpha}^h$ returned by the algorithm becomes, 
\[
 \widehat{\mathrm{CS}}_{n, \alpha}^h =  \left\{\theta \in \mathbb{R} \Bigg| P_{n}(m_{\theta; h} - m_{\widehat{\theta}_1;h}) \leq \sqrt{\frac{6}{n}}\left[\sqrt{\log(1/\alpha)} + 2 \right] \right\} \pm h = \tilde{\mathrm{CS}}_{n,\alpha}^h \pm h. 
 \]
The confidence set $\tilde{\mathrm{CS}}_{n,\alpha}^h$ is contained $P-$almost surely in the following sets,
\begin{equation*}
    \begin{split}
   \tilde{\mathrm{CS}}_{n,\alpha}^h =&  \left\{\theta \in \mathbb{R} \Bigg| P_{n}(m_{\theta; h} - m_{\widehat{\theta}_1;h}) \leq \sqrt{\frac{6}{n}}\left[\sqrt{\log(1/\alpha)} + 2 \right] \right\}  \\
   \subseteq & \left\{\theta \in \mathbb{R} \Bigg| P_n(m_{\theta; h} - m_{\theta_0;h}) \leq \sqrt{\frac{6}{n}}\left[\sqrt{\log(1/\alpha)} + 2 \right] + |P_n(m_{\widehat \theta_1;h} - m_{\theta_0;h})| \right\} \\
   \subseteq & \left\{\theta \in \mathbb{R} \Bigg| P(m_{\theta; h} - m_{\theta_0;h}) \leq -n^{-1/2}G_n(m_{\theta; h} - m_{\theta_0;h}) + \Gamma_n(\widehat \theta_1, \theta_0) \right\} \\
   \subseteq &\left\{\theta \in \mathbb{R} \Bigg| |\theta - \theta_0| \leq \left\{ \left(\frac{4C_0}{(\beta + 1)c_0} \right)^{1/\beta} + 1 \right\}h \right\} \\
   & \bigcup \left\{\theta \in \mathbb{R} \Bigg| |\theta - \theta_0| > \left\{ \left(\frac{4C_0}{(\beta + 1)c_0} \right)^{1/\beta} + 1 \right\}h; \mbox{ } P(m_{\theta; h} - m_{\theta_0;h}) \leq -n^{-1/2}G_n(m_{\theta; h} - m_{\theta_0;h}) + \Gamma_n(\widehat \theta_1, \theta_0) \right\},
    \end{split}
\end{equation*}
where $\Gamma_n(\widehat \theta_1, \theta_0)  = \sqrt{6/n}[\sqrt{\log(1/\alpha)} + 2] + |P_n(m_{\widehat \theta_1;h} - m_{\theta_0;h})|$. We observe that,
\begin{equation}
    \label{eq:pmtheta_theta0}
    P(m_{\theta;h} - m_{\theta_0;h}) = \{ F(\theta_0 + h) - F(\theta_0 - h)\} - \{F(\theta + h) - F(\theta - h)\}. 
\end{equation}
We use assumption~\ref{assump:f} to lower bound $ P(m_{\theta;h} - m_{\theta_0;h})$ for $|\theta - \theta_0| >((4C_0/((\beta + 1)c_0))^{1/\beta} + 1)h$. We note that, 
\begin{equation*}
    \begin{split}
    F(\theta_0 + h) - F(\theta_0 - h) & \geq 2 \int_{0}^h (f(\theta_0) - C_0x^{\beta})dx \\
    &= 2h\left(f(\theta_0) - \frac{C_0h^{\beta}}{\beta + 1} \right).
    \end{split}
\end{equation*}
For $((4C_0/((\beta + 1)c_0))^{1/\beta} + 1)h \leq |\theta - \theta_0| \leq h_0 - h$ we have, 
\begin{equation*}
    \begin{split}
     F(\theta + h) - F(\theta - h)  &\leq \int_{|\theta - \theta_0| - h}^{|\theta - \theta_0| + h}(f(\theta_0) - c_0x^{\beta}) dx \\
     &= f(\theta_0).2h - c_0 \frac{(|\theta - \theta_0| + h)^{\beta + 1} - (|\theta - \theta_0| - h)^{\beta + 1}}{\beta + 1} \\
     &\leq f(\theta_0).2h - c_0 \frac{(\beta + 1)|\theta - \theta_0|^{\beta}h }{\beta + 1} \\
     &= 2h\left(f(\theta_0) - (c_0/2) |\theta - \theta_0|^{\beta} \right). 
    \end{split}
\end{equation*}
Hence for $((4C_0/((\beta + 1)c_0))^{1/\beta} + 1)h \leq |\theta - \theta_0| \leq h_0 - h$ we have, 
\begin{equation*}
    \begin{split}
        P(m_{\theta;h} - m_{\theta_0;h}) =& \{ F(\theta_0 + h) - F(\theta_0 - h)\} - \{F(\theta + h) - F(\theta - h)\} \\
        \geq & 2h\left(f(\theta_0) - \frac{C_0h^{\beta}}{\beta + 1} \right) - 2h\left(f(\theta_0) - (c_0/2) |\theta - \theta_0|^{\beta} \right) \\
        = & 2h \left((c_0/2) |\theta - \theta_0|^{\beta} -  \frac{C_0h^{\beta}}{\beta + 1} \right) \\
        \stackrel{(i)}{\geq} & \frac{hc_0}{2} |\theta - \theta_0|^{\beta}. 
    \end{split}
\end{equation*}
The condition $((4C_0/((\beta + 1)c_0))^{1/\beta} + 1)h \leq |\theta - \theta_0|$ implies that $C_0h^{\beta}/(\beta + 1) \leq (c_0/4)|\theta - \theta_0|^{\beta}$ which justifies step-$(i)$. We now consider $\theta$ such that $h_0 - h < |\theta - \theta_0| \leq h_0 + h $. 
\begin{equation*}
    \begin{split}
     F(\theta + h) - F(\theta - h)  &\leq \int_{|\theta - \theta_0| - h}^{h_0} (f(\theta_0) - c_0x^{\beta}) dx + (f(\theta_0) - c_0h_0^{\beta})(  |\theta - \theta_0| + h - h_0) \\
     &= f(\theta_0)2h - c_0\left\{\frac{h_0^{\beta + 1} - (|\theta - \theta_0| - h)^{\beta + 1}}{\beta + 1} + h_0^{\beta}(  |\theta - \theta_0| + h - h_0) \right\} \\
     &\leq f(\theta_0)2h - c_0\left\{(|\theta - \theta_0| - h)^{\beta}(h_0 - |\theta - \theta_0| + h) + h_0^{\beta}(  |\theta - \theta_0| + h - h_0) \right\} \\
     & \leq 2h \left(f(\theta_0) - c_0(|\theta - \theta_0| - h)^{\beta} \right) \\
     & \leq 2h \left(f(\theta_0) - c_0(h_0 - 2h)^{\beta} \right). 
    \end{split}
\end{equation*}
Moreover if $|\theta - \theta_0| > h_0 + h$ then, 
\[
 F(\theta + h) - F(\theta - h) \leq 2h\left(f(\theta_0) - c_0h_0 ^{\beta} \right). 
\]
Thus combining the above two inequalities we have for any $\theta$ such that $|\theta - \theta_0| > h_0 - h$, 
\begin{equation*}
    \begin{split}
        P(m_{\theta;h} - m_{\theta_0;h}) =& \{ F(\theta_0 + h) - F(\theta_0 - h)\} - \{F(\theta + h) - F(\theta - h)\} \\
        \geq & 2h\left(f(\theta_0) - \frac{C_0h^{\beta}}{\beta + 1} \right) - 2h \left(f(\theta_0) - c_0(h_0 - 2h)^{\beta} \right) \\
        = &2h\left( c_0(h_0 - 2h)^{\beta} -  \frac{C_0h^{\beta}}{\beta + 1} \right). 
    \end{split}
\end{equation*}
These results imply that $P-$almost surely the confidence set $\tilde{\mathrm{CS}}_{n,\alpha}^h$ is contained in, 
\begin{equation*}
    \begin{split}
    \tilde{\mathrm{CS}}_{n,\alpha}^h  \subseteq &  \left\{\theta \in \mathbb{R} \Bigg| |\theta - \theta_0| \leq \left\{ \left(\frac{4C_0}{(\beta + 1)c_0} \right)^{1/\beta} + 1 \right\}h \right\} \\
   & \bigcup \left\{\theta \in \mathbb{R} \Bigg| |\theta - \theta_0| > \left\{ \left(\frac{4C_0}{(\beta + 1)c_0} \right)^{1/\beta} + 1 \right\}h; \mbox{ } P(m_{\theta; h} - m_{\theta_0;h}) \leq -n^{-1/2}G_n(m_{\theta; h} - m_{\theta_0;h}) + \Gamma_n(\widehat \theta_1, \theta_0) \right\} \\
    \subseteq & \left\{\theta \in \mathbb{R} \Bigg| |\theta - \theta_0| \leq \left\{ \left(\frac{4C_0}{(\beta + 1)c_0} \right)^{1/\beta} + 1 \right\}h \right\} \\
    & \bigcup \left\{\theta \in \mathbb{R} \Bigg|   \left\{ \left(\frac{4C_0}{(\beta + 1)c_0} \right)^{1/\beta} + 1 \right\}h < |\theta - \theta_0| \leq h_0 - h; \mbox{ } \frac{hc_0}{2}|\theta - \theta_0|^{\beta} \leq -n^{-1/2}G_n(m_{\theta; h} - m_{\theta_0;h}) + \Gamma_n(\widehat \theta_1, \theta_0) \right\} \\
    & \bigcup \left\{\theta \in \mathbb{R} \Bigg|  |\theta - \theta_0| > h_0 - h; \mbox{ } 2h\left( c_0(h_0 - 2h)^{\beta} -  \frac{C_0h^{\beta}}{\beta + 1} \right) \leq -n^{-1/2}G_n(m_{\theta; h} - m_{\theta_0;h}) + \Gamma_n(\widehat \theta_1, \theta_0) \right\} \\
    \coloneqq & \mathbf{I} \cup \mathbf{II} \cup \mathbf{III}. 
    \end{split}
\end{equation*}
We now derive a bound for $\mathbb{E}[\sup_{|\theta - \theta_0| < \delta}|G_n(m_{\theta;h} - m_{\theta_0;h})|] = \mathbb{E}[\sup_{m \in M_{\theta_0, \delta, h}}|G_nm|] $ where $M_{\theta_0, \delta, h} = \{m_{\theta; h} - m_{\theta_0;h}| |\theta - \theta_0| < \delta\} $ and $0 < \delta < h$. $M_{\theta_0, \delta, h}$ is equipped with envelope function $F = 1$. Moreover for $m = m_{\theta; h} - m_{\theta_0;h} \in M_{\theta_0, \delta, h}$ we have, 
\begin{equation}
\label{eq:m2_bound}
    \begin{split}
        \mathbb{E}[m^2] =& \mathbb{E}(m_{\theta; h} - m_{\theta_0;h})^2 \\
        =& F[[\theta - h, \theta + h] \cap [\theta_0 - h, \theta_0 + h]^c] + F[[\theta - h, \theta + h]^c \cap [\theta_0 - h, \theta_0 + h]] \\
        \leq & \int_{\theta_0+h}^{\theta_0 + h + \delta} (f(\theta_0) - c_0x^{\beta})dx  + \int_{\theta_0-h}^{\theta_0 - h + \delta} (f(\theta_0) - c_0x^{\beta})dx \\
        = & f(\theta_0)\delta - c_0 \frac{(\theta_0 + h+ \delta)^{\beta + 1} - (\theta_0 + h)^{\beta + 1}}{\beta + 1} + f(\theta_0)\delta - c_0 \frac{(\theta_0 - h+ \delta)^{\beta + 1} - (\theta_0 - h)^{\beta + 1}}{\beta + 1} \\
        \leq & \delta\left(f(\theta_0) - c_0(\theta_0 + h)^{\beta} \right) + \delta\left(f(\theta_0) - c_0(\theta_0 - h)^{\beta} \right) \\
        \leq & 2f(\theta_0)\delta \\
        \coloneqq& C_1 \delta. 
    \end{split}
\end{equation}
Thus we see for any $m \in M_{\theta_0, \delta, h}$, $\mathbb{E}[m^2] \leq t^2\mathbb{E}[F^2] = t^2.1$ where $t = \sqrt{C_1\delta}$. It can be easily seen that for any $\varepsilon >0$, the $\varepsilon-$ bracketing covering number $\mathcal{N}_{[]}(\varepsilon, M_{\theta_0, \delta, h}, |\cdot|)  \leq 3$. Thus the bracketing entropy integral is, 
\[
J_{[]}(\delta, M_{\theta_0, \delta, h}, |\cdot|) = \int_{0}^{\delta} \sqrt{1 + \log(\mathcal{N}_{[]}(\varepsilon, M_{\theta_0, \delta, h}, |\cdot|))} d\varepsilon \leq \sqrt{1 + \log3} \delta \coloneqq C_2 \delta. 
\]
Using Theorem-$2.14.17$ of \cite{wellner2013weak} we have, 
\begin{equation}
\label{eq:gn_bound_local}
    \begin{split}
    \mathbb{E}[\sup_{m \in M_{\theta_0, \delta, h}}|G_nm|]  \leq & J_{[]}(t, M_{\theta_0, \delta, h}, |\cdot|) \left(1 + \frac{J_{[]}(t, M_{\theta_0, \delta, h}, |\cdot|)}{t^2 \sqrt{n}} \right)  \\
    \leq & C_2t \left(1 + \frac{C_2}{t \sqrt{n}}
    \right) \\
    \leq & C_2\sqrt{C_1\delta} + \frac{C_2^2}{\sqrt{n}} \\
    \coloneqq & \phi_n(\delta). 
    \end{split}
\end{equation}
We note that $\phi_n(x)/x^q$ is a non-increasing function of $x$ for any $q > 1/2$. Following the above steps, one can also show that, 
\begin{equation}
\label{eq:gn_bound_global}
    \begin{split}
        \mathbb{E}[\sup_{\theta \in \mathbb{R}}|G_n(m_{\theta;h} - m_{\theta_0;h})|] \leq C_2 + \frac{C_2^2}{\sqrt{n}} \leq C_3 \mbox{ say}.
    \end{split}
\end{equation}
The above holds because of the fact that $\mathbb{E}[(m_{\theta;h} - m_{\theta_0;h})^2] \leq 1$. Hence instead of $t = \sqrt{C_1\delta}$ we use $t = 1$ in the earlier proof to get the result. We now obtain a concentration inequality for $\Gamma_n(\widehat \theta_1, \theta_0)$ i.e. we find $\gamma_{n,\epsilon} > 0$ such that $\mathbb{P}(\Gamma_n(\widehat \theta_1, \theta_0) > \gamma_{n,\epsilon})  \leq \epsilon/2$. Consider the following definition of $\tilde{\gamma}_{n,\epsilon}$, 
\[
\tilde \gamma_{n,\epsilon} = \sqrt{\frac{6}{n}}\left[\sqrt{\log(1/\alpha)} + \sqrt{\log(4/\epsilon)} + 4 \right] + |P(m_{\widehat \theta_1;h} - m_{\theta_0;h})|
\]
We denote by $\mathbb{P}^1$ the conditional probability $\mathbb{P}(\cdot |S_1)$. We can bound $\mathbb{P}^1(\Gamma_n(\widehat \theta_1, \theta_0) > \tilde \gamma_{n,\epsilon})$ as follows, 
\begin{equation*}
    \begin{split}
\mathbb{P}^1(\Gamma_n(\widehat \theta_1, \theta_0) > \tilde \gamma_{n,\epsilon}) =& \mathbb{P}^1\left(\sqrt{\frac{6}{n}}\left[\sqrt{\log(1/\alpha)} + 2 \right] + |P_n(m_{\widehat \theta_1;h} - m_{\theta_0;h})| > \gamma_{n,\epsilon} \right) \\
\leq & \mathbb{P}^1\left( |(P_n - P)(m_{\widehat \theta_1;h} - m_{\theta_0;h})| > \tilde \gamma_{n,\epsilon} - \sqrt{\frac{6}{n}}\left[\sqrt{\log(1/\alpha)} + 2 \right] - |P(m_{\widehat \theta_1;h} - m_{\theta_0;h})| \right)  \\
= & \mathbb{P}^1 \left( |(P_n - P)(m_{\widehat \theta_1;h} - m_{\theta_0;h})| >  \sqrt{\frac{6}{n}}\left[\sqrt{\log(4/\epsilon)} + 2 \right]  \right) \\
\stackrel{(i)}{\leq} &  \epsilon/4.  
    \end{split}
\end{equation*}
The step-$(i)$ follows from \Cref{lem:trinomial}. From \eqref{eq:m2_bound} we know that if $|\widehat \theta_1 - \theta_0| \leq h$ then, 
\begin{equation*}
    \begin{split}
  & |P(m_{\widehat \theta_1;h} - m_{\theta_0;h})|\\  
  =& | F(\widehat \theta_1 + h) - F(\widehat \theta_1 - h) - F(\theta_0 + h) + F(\theta_0 - h) |\\
  \leq & |F(\widehat \theta_1 + h) - F(\theta_0 + h)| + | F(\widehat \theta_1 - h) -  F(\theta_0 - h)| \\
  \leq & 2f(\theta_0)|\widehat \theta_1 - \theta_0|
    \end{split}
\end{equation*}
We know that with probability greater than or equal to $1 - \epsilon/4$, $|\widehat \theta_1 - \theta_0| \leq c_{\epsilon}\tilde s_n$. Thus with probability at-least $1 - \epsilon/4$, 
\begin{equation*}
    \begin{split}
        \tilde \gamma_{n,\epsilon} =& \sqrt{\frac{6}{n}}\left[\sqrt{\log(1/\alpha)} + \sqrt{\log(4/\epsilon)} + 4 \right] + |P(m_{\widehat \theta_1;h} - m_{\theta_0;h})| \\
        = & \sqrt{\frac{6}{n}}\left[\sqrt{\log(1/\alpha)} + \sqrt{\log(4/\epsilon)} + 4 \right] + |P(m_{\widehat \theta_1;h} - m_{\theta_0;h})|\textbf{1}\{|\widehat \theta_1 - \theta_0| \leq h\} + |P(m_{\widehat \theta_1;h} - m_{\theta_0;h})|\textbf{1}\{|\widehat \theta_1 - \theta_0| > h\} \\
        \leq & \sqrt{\frac{6}{n}}\left[\sqrt{\log(1/\alpha)} + \sqrt{\log(4/\epsilon)} + 4 \right] + 2f(\theta_0)|\widehat \theta_1 - \theta_0|\textbf{1}\{|\widehat \theta_1 - \theta_0| \leq h\} + |P(m_{\widehat \theta_1;h} - m_{\theta_0;h})|\textbf{1}\{|\widehat \theta_1 - \theta_0| > h\} \\
       \leq & \sqrt{\frac{6}{n}}\left[\sqrt{\log(1/\alpha)} + \sqrt{\log(4/\epsilon)} + 4 \right] + 2f(\theta_0)c_{\epsilon}\tilde s_n\textbf{1}\{|\widehat \theta_1 - \theta_0| \leq h\} + |P(m_{\widehat \theta_1;h} - m_{\theta_0;h})|\textbf{1}\{|\widehat \theta_1 - \theta_0| > h\} \\ 
       \leq & C_{\epsilon}s_n + |P(m_{\widehat \theta_1;h} - m_{\theta_0;h})|\textbf{1}\{|\widehat \theta_1 - \theta_0| > h\} \quad \mbox{for a suitable }C_{\epsilon}, s_n > 0 \\
       \coloneqq & \gamma_{n, \epsilon}. 
    \end{split}
\end{equation*}
Combining the two statements we obtain that, 
\begin{equation}
    \label{eq:gamma_n_bound}
    \mathbb{P}(\Gamma_n(\widehat \theta_1, \theta_0) \leq \gamma_{n,\epsilon}) \geq 2(1 - (\epsilon/4)) - 1 \geq 1 - (\epsilon/2). 
\end{equation}

Let us first analyze $\mathbf{III}$. We show that $\mathbb{P}(|\mathbf{III}| \neq \phi) \leq 2\epsilon/3$ under the assumption that,
\begin{equation}
    \label{eq:asump_epsilon}
    h\left( c_0(h_0 - 2h)^{\beta} -  \frac{C_0h^{\beta}}{\beta + 1} \right) > \max\left\{\gamma_{n,\epsilon}, \frac{6C_3}{\sqrt{n}\epsilon} \right\}.
\end{equation}
We observe that, 
\begin{equation*}
    \begin{split}
      &\mathbb{P}^1 \left(\exists \theta \in \mathbb{R} \Bigg|  |\theta - \theta_0| > h_0 - h; \mbox{ } 2h\left( c_0(h_0 - 2h)^{\beta} -  \frac{C_0h^{\beta}}{\beta + 1} \right) \leq -n^{-1/2}G_n(m_{\theta; h} - m_{\theta_0;h}) + \Gamma_n(\widehat \theta_1, \theta_0) \right) \\
   \leq  &\mathbb{P}^1 \left(\exists \theta \in \mathbb{R} \Bigg|  |\theta - \theta_0| > h_0 - h; \mbox{ } h\left( c_0(h_0 - 2h)^{\beta} -  \frac{C_0h^{\beta}}{\beta + 1} \right) \leq -n^{-1/2}G_n(m_{\theta; h} - m_{\theta_0;h}) \right) \\
   & + \mathbb{P}^1 \left(\exists \theta \in \mathbb{R} \Bigg|  |\theta - \theta_0| > h_0 - h; \mbox{ } h\left( c_0(h_0 - 2h)^{\beta} -  \frac{C_0h^{\beta}}{\beta + 1} \right) \leq  \Gamma_n(\widehat \theta_1, \theta_0) \right).
    \end{split}
\end{equation*}
Let us handle the first term at first. We observe that, 
\begin{equation*}
    \begin{split}
 & \mathbb{P}^1 \left(\exists \theta \in \mathbb{R} \Bigg|  |\theta - \theta_0| > h_0 - h; \mbox{ } h\left( c_0(h_0 - 2h)^{\beta} -  \frac{C_0h^{\beta}}{\beta + 1} \right) \leq -n^{-1/2}G_n(m_{\theta; h} - m_{\theta_0;h}) \right) \\
\leq &   \mathbb{P}^1 \left( h\left( c_0(h_0 - 2h)^{\beta} -  \frac{C_0h^{\beta}}{\beta + 1} \right) \leq n^{-1/2} \sup_{\theta \in \mathbb{R}}|G_n(m_{\theta; h} - m_{\theta_0;h}|) \right) \\
\stackrel{\eqref{eq:asump_epsilon}}{\leq} & \mathbb{P}^1 \left(\frac{6C_3}{\sqrt{n}\epsilon} \leq n^{-1/2} \sup_{\theta \in \mathbb{R}}|G_n(m_{\theta; h} - m_{\theta_0;h}|) \right) \\
\stackrel{(i)}{\leq} & \frac{\sqrt{n}\epsilon}{6C_3} n^{-1/2}\mathbb{E}[\sup_{\theta \in \mathbb{R}}|G_n(m_{\theta; h} - m_{\theta_0;h}|] \\
\stackrel{\eqref{eq:gn_bound_global}}{\leq} & \frac{\epsilon}{6C_3} C_3 \\
=& \epsilon/6. 
    \end{split}
\end{equation*}
Here step-$(i)$ follows using Markov's inequality. Under the event $\{\gamma_{n,\epsilon} \geq \tilde \gamma_{n,\epsilon}\}$ (probability of occurence at-least $ 1- (\epsilon/4)$) we can handle the second term as follows, 
\begin{equation*}
    \begin{split}
 &  \mathbb{P}^1 \left(\exists \theta \in \mathbb{R} \Bigg|  |\theta - \theta_0| > h_0 - h; \mbox{ } h\left( c_0(h_0 - 2h)^{\beta} -  \frac{C_0h^{\beta}}{\beta + 1} \right) \leq  \Gamma_n(\widehat \theta_1, \theta_0) \right) \\
 \leq &  \mathbb{P}^1 \left(\left\{\exists \theta \in \mathbb{R} \Bigg|  |\theta - \theta_0| > h_0 - h; \mbox{ } h\left( c_0(h_0 - 2h)^{\beta} -  \frac{C_0h^{\beta}}{\beta + 1} \right) \leq  \Gamma_n(\widehat \theta_1, \theta_0)\right\} \cap \{\Gamma_n(\widehat \theta_1, \theta_0) \leq \tilde \gamma_{n,\epsilon} \} \right) \\
 &+ \mathbb{P}^1\left(\Gamma_n(\widehat \theta_1, \theta_0) > \tilde \gamma_{n,\epsilon} \right) \\
 \stackrel{\eqref{eq:asump_epsilon}}{\leq} & 0 + \epsilon/4 \\
 =&\epsilon/4. 
    \end{split}
\end{equation*}
Taking intersection of this event with $\{\gamma_{n,\epsilon} \geq \tilde \gamma_{n,\epsilon}\}$ we have, 
\[
\mathbb{P} \left(\exists \theta \in \mathbb{R} \Bigg|  |\theta - \theta_0| > h_0 - h; \mbox{ } h\left( c_0(h_0 - 2h)^{\beta} -  \frac{C_0h^{\beta}}{\beta + 1} \right) \leq  \Gamma_n(\widehat \theta_1, \theta_0) \right) \leq \epsilon/2. 
\]
Combining the bounds for the first and second term we get, 
\begin{equation*}
    \begin{split}
   & \mathbb{P}(|\mathbf{III}| \neq \phi) \\
   \leq & \mathbb{P} \left(\exists \theta \in \mathbb{R} \Bigg|  |\theta - \theta_0| > h_0 - h; \mbox{ } h\left( c_0(h_0 - 2h)^{\beta} -  \frac{C_0h^{\beta}}{\beta + 1} \right) \leq -n^{-1/2}G_n(m_{\theta; h} - m_{\theta_0;h}) \right)  \\
   \leq & (\epsilon/6) + (\epsilon/2) \\
   =& 2\epsilon/3. 
    \end{split}
\end{equation*}
We now analyze the width of the set $\mathbf{II}$. Let $r_n$ be any value that satisfies $r_n^{-2} \phi_n((hc_0/2)^{-1/\beta} r_n^{2/\beta})\leq n^{1/2} $. We compute such an $r_n$ as follows,
\begin{equation*}
    \begin{split}
  & r_n^{-2} \phi_n((hc_0/2)^{-1/\beta} r_n^{2/\beta})\leq n^{1/2} \\
  \implies & r_n^{-2} \left\{C_2 \sqrt{C_1 (hc_0/2)^{-1/\beta} r_n^{2/\beta}} + \frac{C_2^2}{\sqrt{n}} \right\} \leq n^{1/2} \\
  \implies & r_n^{-(2\beta - 1)/\beta} C_2\sqrt{C_1 (hc_0/2)^{-1/\beta}} + \frac{r_n^{-2}C_2^2}{\sqrt{n}} \leq n^{1/2} . 
    \end{split}
\end{equation*}
Thus if we take,
\[
r_n = \max\left\{\left[2C_2\sqrt{C_1 (hc_0/2)^{-1/\beta}}\right]^{\beta/(2\beta - 1)} n^{-\beta/(4\beta - 2)} , \sqrt{2}C_2 n^{-1/2} \right\},
\]
the condition $r_n^{-2} \phi_n((hc_0/2)^{-1/\beta} r_n^{2/\beta})\leq n^{1/2} $ is satisfied. Let us define $R_M$ as follows, 
\[
R_M = 2^{M/\beta} (hc_0/2)^{-1/\beta}\{r_n^{2/\beta} +  \Gamma_n(\widehat \theta_1, \theta_0)^{1/\beta} \}.
\]
We show that for $M$ large enough, $\mathbb{P}(|\mathbf{II} \cap \{|\theta - \theta_0| > R_M\}| \neq \phi \} < \epsilon/6$. We observe that, 
\begin{equation*}
    \begin{split}
      & \mathbb{P}^1\left(\exists\theta \in \mathbb{R} \Bigg|   R_M< |\theta - \theta_0| \leq h_0 - h; \mbox{ } \frac{hc_0}{2}|\theta - \theta_0|^{\beta} \leq -n^{-1/2}G_n(m_{\theta; h} - m_{\theta_0;h}) + \Gamma_n(\widehat \theta_1, \theta_0) \right) \\
      \leq & \mathbb{P}^1\left(\exists \theta \in \mathbb{R} \Bigg|    R_M < |\theta - \theta_0| \leq h_0 - h; \mbox{ } \frac{hc_0}{4}|\theta - \theta_0|^{\beta} \leq -n^{-1/2}G_n(m_{\theta; h} - m_{\theta_0;h}) \right) \\
      &+ \mathbb{P}^1\left(\exists \theta \in \mathbb{R} \Bigg| R_M< |\theta - \theta_0| \leq h_0 - h; \mbox{ } \frac{hc_0}{4}|\theta - \theta_0|^{\beta} \leq  \Gamma_n(\widehat \theta_1, \theta_0) \right).
    \end{split}
\end{equation*}
We start by providing bound for the first term in the decomposition. We define the shells $J_i$ for $i \in \mathbb{N}$, 
\[
J_i \coloneqq \{\theta \in \mathbb{R}| 2^{i/\beta} (hc_0/2)^{-1/\beta}r_n^{2/\beta} \leq | \theta - \theta_0| \leq 2^{(i+1)/\beta}(hc_0/2)^{-1/\beta}r_n^{2/\beta} \}.
\]
We observe that, 
\begin{equation*}
    \begin{split}
      & \mathbb{P}^1\left(\exists \theta \in \mathbb{R} \Bigg|    R_M < |\theta - \theta_0| \leq h_0 - h; \mbox{ } \frac{hc_0}{4}|\theta - \theta_0|^{\beta} \leq -n^{-1/2}G_n(m_{\theta; h} - m_{\theta_0;h}) \right) \\
      \leq & \mathbb{P}^1\left(\exists \theta \in \mathbb{R} \Bigg|    2^{M/\beta} (hc_0/2)^{-1/\beta}r_n^{2/\beta} < |\theta - \theta_0| \leq h_0 - h; \mbox{ } \frac{hc_0}{4}|\theta - \theta_0|^{\beta} \leq -n^{-1/2}G_n(m_{\theta; h} - m_{\theta_0;h}) \right) \\
      \leq & \mathbb{P}^1\left(\exists \theta \in J_i \mbox{ for } i\geq M \Bigg|  \frac{hc_0}{4}|\theta - \theta_0|^{\beta} \leq n^{-1/2}|G_n(m_{\theta; h} - m_{\theta_0;h})|  \right) \\
      \leq & \sum_{i = M}^{\infty} \mathbb{P}^1\left(\exists \theta \in J_i  \Bigg|  \frac{hc_0}{4}|\theta - \theta_0|^{\beta} \leq n^{-1/2}|G_n(m_{\theta; h} - m_{\theta_0;h})|  \right) \\
      \leq & \sum_{i = M}^{\infty} \mathbb{P}^1\left(\frac{hc_0}{4} 2^i (hc_0/2)^{-1}r_n^2 \leq n^{-1/2} \sup_{\theta \in J_i} |G_n(m_{\theta; h} - m_{\theta_0;h})|  \right) \\
      \stackrel{(i)}{\leq} & 2 \sum_{i = M}^{\infty} 2^{-i}r_n^{-2} n^{-1/2}\mathbb{E}\left[\sup_{\theta \in J_i} |G_n(m_{\theta; h} - m_{\theta_0;h})| \right] \\
      \stackrel{\eqref{eq:gn_bound_local}}{\leq} & 2 \sum_{i = M}^{\infty} 2^{-i}r_n^{-2} \phi_n(2^{(i+1)/\beta}(hc_0/2)^{-1/\beta}r_n^{2/\beta})  \\
      \stackrel{(ii)}{\leq} & 2 \sum_{i = M}^{\infty} 2^{-i} 2^{(i+1)q/\beta} n^{-1/2} r_n^{-2} \phi_n((hc_0/2)^{-1/\beta}r_n^{2/\beta}) \quad \mbox{for } 1/2 < q < \beta\\
      \stackrel{(iii)}{\leq} & 2^{1 + (q/\beta)} \sum_{i = M}^{\infty} 2^{-i(\beta - q)/\beta} \\
      \leq & \epsilon/6 \quad \mbox{for large enough }M. 
    \end{split}
\end{equation*}
Step-$(i)$ uses Markov's inequality. In step-$(ii)$ we use the fact that $\phi(x)/x^q$ is a non-increasing function of $x$ for any $q > 1/2$. Step-$(iii)$ follows from the definition of $r_n^2$. Now we analyze the second term in the decomposition under the event $\{\gamma_{n,\epsilon} \geq \tilde \gamma_{n,\epsilon}\}$. 
\begin{equation*}
    \begin{split}
        & \mathbb{P}^1\left(\exists \theta \in \mathbb{R} \Bigg|    R_M < |\theta - \theta_0| \leq h_0 - h; \mbox{ } \frac{hc_0}{4}|\theta - \theta_0|^{\beta}  \leq \Gamma_n(\widehat \theta_1, \theta_0) \right) \\
        \leq & \mathbb{P}^1\left(  R_M^{\beta}(c_0/4)h \leq   \Gamma_n(\widehat \theta_1, \theta_0) \right) \\
        \leq & \mathbb{P}^1\left(   \frac{hc_0}{4}2^M (hc_0/2)^{-1}\Gamma_n(\widehat \theta_1, \theta_0) \leq   \Gamma_n(\widehat \theta_1, \theta_0) \right) \\
        \leq & \mathbb{P}^1\left(  2^{M-1} \leq   1\right) \\
         = & 0. 
    \end{split}
\end{equation*}
Combining the above two bounds we have the following for large enough $M > 1$, 
\begin{equation*}
    \begin{split}
   & \mathbb{P}(|\mathbf{II} \cap \{|\theta - \theta_0| > R_M| \neq \phi \} \\
  \leq & \mathbb{P}\left(\exists\theta \in \mathbb{R} \Bigg|    R_M < |\theta - \theta_0| \leq h_0 - h; \mbox{ } \frac{hc_0}{2}|\theta - \theta_0|^{\beta} \leq -n^{-1/2}G_n(m_{\theta; h} - m_{\theta_0;h}) + \Gamma_n(\widehat \theta_1, \theta_0) \right) \\
  \leq &\epsilon/6. 
    \end{split}
\end{equation*}
Therefore with probability at-least $ 1- (\epsilon/6)$, $\mathbf{II} \subseteq \{\theta| |\theta - \theta_0| \leq R_M\}$. Combining this with the result we have obtained for the set $\mathbf{III}$ we can say that with probability at-least $1 - \epsilon$, 
\begin{equation*}
    \begin{split}
         \tilde{\mathrm{CS}}_{n,\alpha}^h  \subseteq &  \left\{\theta \in \mathbb{R} \Bigg| |\theta - \theta_0| \leq \left\{ \left(\frac{4C_0}{(\beta + 1)c_0} \right)^{1/\beta} + 1 \right\}h \right\}  \bigcup \mathbf{II} \bigcup \mathbf{III} \\
         \subseteq &  \left\{\theta \in \mathbb{R} \Bigg| |\theta - \theta_0| \leq \left\{ \left(\frac{4C_0}{(\beta + 1)c_0} \right)^{1/\beta} + 1 \right\}h \right\}  \bigcup \{\theta| |\theta - \theta_0| \leq R_M\}. 
    \end{split}
\end{equation*}
Therefore with probability at-least $ 1 - \epsilon$ we can bound the width of the confidence set $\widehat{\mathrm{CS}}^h_{n,\alpha}$ as follows, 
\begin{equation*}
    \begin{split}
       & \mbox{Width}\left( \widehat{\mathrm{CS}}^h_{n,\alpha} \right) \\
      \leq &  \mbox{Width}\left( \tilde{\mathrm{CS}}^h_{n,\alpha} \right) + 2h \\
      \leq & \left\{2 \left(\frac{4C_0}{(\beta + 1)c_0} \right)^{1/\beta} + 4 \right\}h + 2 R_M \\
      \leq &\left\{2 \left(\frac{4C_0}{(\beta + 1)c_0} \right)^{1/\beta} + 4 \right\}h + \mathcal{C} \left(\frac{2}{hc_0}\right)^{1/\beta} (r_n^{2/\beta} + \Gamma_n(\widehat \theta_1, \theta_0)^{1/\beta}) \\
      \leq & \left\{ 2\left(\frac{4C_0}{(\beta + 1)c_0} \right)^{1/\beta} + 4 \right\}h \\
      & + \mathcal{C} \left(\frac{2}{hc_0}\right)^{1/\beta}  \max\left\{\left[2C_2\sqrt{C_1 (hc_0/2)^{-1/\beta}}\right]^{2/(2\beta - 1)} n^{-1/(2\beta - 1)} , (\sqrt{2}C_2)^{2/\beta} n^{-\beta} \right\}  \\
      & + \mathcal{C} \left(\frac{2}{hc_0}\right)^{1/\beta} \left\{C_{\epsilon}s_n + |P(m_{\widehat \theta_1;h} - m_{\theta_0;h})|\textbf{1}\{|\widehat \theta_1 - \theta_0| > h\} \right\} ^{1/\beta}. 
    \end{split} 
\end{equation*}
In the above computation we use $\mathcal{C}$ to denote the constant $2^{1 + (M/\beta)}$. 

In \Cref{thm:m_estimator_width} we have $m = 2n$ i.e.\ the sets $S_1$ and $S_2$ are of the same cardinality. We know that the minimax rate of convergence of the estimator $\widehat \theta_1$ under assumption~\ref{assump:f} is $n^{1/(1 + 2\beta)}$. See for instance \cite{arias2022estimation}, \cite{dasgupta2014optimal}. In particular from Theorem-$1$ in \cite{arias2022estimation} we know that under assumption~\ref{assump:f} there exists a constant $A >0$ (depending on the constants in assumption~\ref{assump:f}) such that with probability at-least $1 - (\epsilon/4)$ we have $|\widehat \theta_1 - \theta_0| \leq A\log(4/\epsilon)n^{-1/(1 + 2\beta)} = A_{\epsilon} n^{-1/(1 + 2\beta)}$ (say). We assume $\beta \geq 1$ for this part. This implies that with probability at-least $1 - \epsilon$ we have, 
\begin{equation*}
    \begin{split}
         & \mbox{Width}\left( \widehat{\mathrm{CS}}^h_{n,\alpha} \right) \\
      \leq &  \left\{2 \left(\frac{4C_0}{(\beta + 1)c_0} \right)^{1/\beta} + 4 \right\}h \\
      & + \mathcal{C} \left(\frac{2}{hc_0}\right)^{1/\beta}  \max\left\{\left[2C_2\sqrt{C_1 (hc_0/2)^{-1/\beta}}\right]^{2/(2\beta - 1)} n^{-1/(2\beta - 1)} , (\sqrt{2}C_2)^{2/\beta} n^{-\beta} \right\}  \\
      & + \mathcal{C} \left(\frac{2}{hc_0}\right)^{1/\beta} \left\{C_{\epsilon}s_n + |P(m_{\widehat \theta_1;h} - m_{\theta_0;h})|\textbf{1}\{|\widehat \theta_1 - \theta_0| > h\} \right\} ^{1/\beta} \\
      \leq & \left\{ 2\left(\frac{4C_0}{(\beta + 1)c_0} \right)^{1/\beta} + 4 \right\}h \\ 
      & + \mathcal{C}(2/c_0)^{1/\beta} \left[2C_2\sqrt{C_1 (c_0/2)^{-1/\beta}}\right]^{2/(2\beta - 1)} n^{-1/(2\beta - 1)} h^{-2/(2\beta -1)} + \mathcal{C}(2/c_0)^{1/\beta} (\sqrt{2}C_2)^{2/\beta} n^{-\beta}h^{-1/\beta} \\
      &+ \mathcal{C}(2/c_0)^{1/\beta}h^{-1/\beta} \left\{\sqrt{\frac{6}{n}}\left[\sqrt{\log(1/\alpha)} + \sqrt{\log(4/\epsilon)} + 4 \right] + 2f(\theta_0)A_{\epsilon}n^{-1/(1+ 2\beta)}\textbf{1}\{|\widehat \theta_1 - \theta_0| \leq h\} \right\}^{1/\beta} \\
     & + \mathcal{C}(2/c_0)^{1/\beta}h^{-1/\beta} \left\{|P(m_{\widehat \theta_1;h} - m_{\theta_0;h})|\textbf{1}\{|\widehat \theta_1 - \theta_0| > h\}\right\}^{1/\beta}.
    \end{split}
\end{equation*}
Let $\mathfrak{C} > 0$ be such that $\sqrt{6/n}[\sqrt{\log(1/\alpha)} + \sqrt{\log(4/\epsilon)} + 4 ] \leq (\mathfrak{C} - 2f(\theta_0))A_{\epsilon}n^{-1/(1+ 2\beta)}$. Using $h = (\log n)^kn^{-1/(1 + 2\beta)}$ (for any $k > 0$) we have with probability at-least $1 - \epsilon$, 
% \begin{equation*}
%     \begin{split}
%     & \mbox{Width}\left( \widehat{\mathrm{CS}}^h_{n,\alpha} \right) \\
%     \leq & \left\{ \left(\frac{4C_0}{(\beta + 1)c_0} \right)^{1/\beta} + 3 \right\}A_{\epsilon}n^{-1/(1 + 2\beta)} \\
%     &+ \mathcal{C}(2/c_0)^{1/\beta} \left[2C_2\sqrt{C_1 (c_0/2)^{-1/\beta}}\right]^{2/(2\beta - 1)}  A_{\epsilon}^{-2/(2\beta -1)} n^{-1/(1 + 2\beta)} \\
%     &+ \mathcal{C}(2/c_0)^{1/\beta} (\sqrt{2}C_2)^{2/\beta} A_{\epsilon}^{-1/\beta} n^{-\beta + (1/(\beta(1+2\beta)))}\\
%     &+ \mathcal{C}(2/c_0)^{1/\beta}A_{\epsilon}^{-1/\beta}n^{1/(\beta(1 + 2\beta))}\mathfrak{C}^{1/\beta}A_{\epsilon}^{1/\beta} n^{-1/(\beta(1 + 2\beta))}. 
%     \end{split}
% \end{equation*}
\begin{equation*}
    \begin{split}
    & \mbox{Width}\left( \widehat{\mathrm{CS}}^h_{n,\alpha} \right) \\
    \leq & \left\{ 2\left(\frac{4C_0}{(\beta + 1)c_0} \right)^{1/\beta} + 4 \right\} (\log n)^k n^{-1/(1 + 2\beta)} \\
    &+ \mathcal{C}(2/c_0)^{1/\beta} \left[2C_2\sqrt{C_1 (c_0/2)^{-1/\beta}}\right]^{2/(2\beta - 1)}  (\log n)^{-2k/(2\beta -1)} n^{-1/(1 + 2\beta)} \\
    &+ \mathcal{C}(2/c_0)^{1/\beta} (\sqrt{2}C_2)^{2/\beta} (\log n)^{-k/\beta} n^{-\beta + (1/(\beta(1+2\beta)))}\\
    &+ \mathcal{C}(2/c_0)^{1/\beta}(\log n)^{-k/\beta}n^{1/(\beta(1 + 2\beta))}\mathfrak{C}^{1/\beta}A_{\epsilon}^{1/\beta} n^{-1/(\beta(1 + 2\beta))}. 
    \end{split}
\end{equation*}
Thus with probability greater than or equal to $1 - \epsilon$ we have the following for any $k'> k$, 
\begin{equation*}
    \begin{split}
       & n^{1/(1+ 2\beta)}(\log n)^{-k'} \mbox{Width}\left( \widehat{\mathrm{CS}}^h_{n,\alpha} \right)  \\
        \leq & \left\{ 2\left(\frac{4C_0}{(\beta + 1)c_0} \right)^{1/\beta} + 4 \right\}(\log n)^{-(k' - k)} + \mathcal{C}(2/c_0)^{1/\beta} \left[2C_2\sqrt{C_1 (c_0/2)^{-1/\beta}}\right]^{2/(2\beta - 1)}  (\log n)^{-k'-(2k/(2\beta -1))} \\
        &+ \mathcal{C}(2/c_0)^{1/\beta} (\sqrt{2}C_2)^{2/\beta} (\log n)^{-k'-(k/\beta)} n^{-\beta + (1/(\beta(1+2\beta))) + (1/(1+ 2\beta))} \\
        &+ \mathcal{C}(2/c_0)^{1/\beta}\mathfrak{C}^{1/\beta}A_{\epsilon}^{1/\beta}(\log n)^{-k'-(k/\beta)}. 
    \end{split}
\end{equation*}
It can be checked that for $\beta \geq 1$, the exponent $-\beta + (1/(\beta(1+2\beta))) + (1/(1+ 2\beta))$ stays negative. Thus for appropriate choice of $h$ see that for any $k' > 0$ $n^{1/(1+ 2\beta)}(\log n)^{-k'} \mbox{Width}( \widehat{\mathrm{CS}}^h_{n,\alpha} ) = O_P(1)$ as $n \rightarrow \infty$. In fact the proof can be easily modified to get the following generalisation. Let $\ell(n)$ be an increasing function of $n$ such that $\ell(n) \rightarrow \infty$ as $n \rightarrow \infty$ and $\ell(n) = o(n^{\tau})$ as $n \rightarrow \infty$ for any $\tau > 0$. Then setting $h = n^{-1/(1 + 2\beta)}\ell(n)^{1/2}$ we have, 
\[
n^{1/(1+2\beta)}\ell(n)^{-1}\mbox{Width}\left( \widehat{\mathrm{CS}}^h_{n,\alpha} \right)  = O_P(1) \quad \mbox{as} \quad n \rightarrow \infty. 
\]
This completes the proof of the theorem. 

\section{Proof of \Cref{thm:ada_m_estimation}}
\label{app:ada_M_estimation}
We follow the same notations as in the proof of \Cref{thm:m_estimation}. We recall that if $\theta_0 = \mbox{arg max}_x f(x)$ and $\theta_h = \mbox{arg max}_y f*g_h(y) = \mbox{arg min}_y Pm_{y;h}$ then, 
\[
\theta_0 \in (\theta_h - h, \theta_h + h). 
\]
Thus if we can compute a finite sample valid $(1 - \alpha)$ confidence set $\tilde{\mathrm{CI}}_{n, \alpha}^h$ of $\theta_h$, then $\tilde{\mathrm{CI}}_{n, \alpha}^h \pm h$ will be a finite sample valid $(1 - \alpha)$ confidence set of $\theta_0$. We use the famous DKW inequality [\cite{massart1990tight}] for this purpose, 
\[
\mathbb{P}\left(\sup_{x \in \mathbb{R}}|P_n((-\infty, x]) - P((-\infty, x])| \leq \sqrt{\frac{\log(2/\alpha)}{2n}} \right) \geq 1 - \alpha. 
\]
This implies that if $\mathcal{I} = \{ (c, d]: -\infty \leq c \leq d \leq \infty \}$ then, 
\begin{equation}
    \label{eq:simul_DKW}
    \mathbb{P}\left(\sup_{I \in \mathcal{I}}|P_n(I) - P(I)| \leq \sqrt{\frac{2\log(2/\alpha)}{n}} \right) \geq 1 - \alpha.
\end{equation}
% Here we use Lemma-$6.3.2$ of \cite{reiss2012approximate} to obtain a simultaneously valid $(1 - \alpha)$ confidence set for all $\theta_h$ ($h > 0$), 
% \[
% \mathbb{P}\left\{\sup_{I \in \mathcal{I}} \frac{\sqrt{n} |P_n(I) - P(I)|}{\max\{ \sigma(I), \sqrt{\log(n+2)/ n} \}} \geq \epsilon \right\} \leq (n + 2)^2 \exp\left[-\epsilon \rho + \frac{3}{4} \rho^2 + \frac{13}{2} \right] \mbox{ for } \epsilon, \rho > 0,
% \]
% where $\mathcal{I}$ is the set of all intervals on $\mathbb{R}$ and $\sigma(I) = \sqrt{P(I)(1 - P(I))} \leq 1/2$. In the above result the notations $P_n(I), P(I)$ are used to denote $P_n(\textbf{1}\{X \in I\}), P(\textbf{1}\{X \in I\})$ respectively. We set $\rho = \sqrt{\log(n + 2)}$, $\epsilon = (11/4)\sqrt{\log(n+2)} + (\log(1/\alpha) + (13/2)) /\sqrt{\log(n+2)}$ in the above bound,
% \begin{equation}
%     \label{eq:simul_reiss}
%     \mathbb{P}\left\{\sup_{I \in \mathcal{I}}  |P_n(I) - P(I)| < \frac{11}{8} \sqrt{\frac{\log(n+2)}{n}} + \frac{(\log(1/\alpha)/2) + (13/4)}{\sqrt{n \log(n+2)}}\right\} \geq 1 - \alpha. 
% \end{equation}
This implies that, 
% \begin{equation}
%     \label{eq:sumul_reiss_2}
%     \begin{split}
%         &\mathbb{P}\left\{(P_n - P)(m_{ \theta_h;h} - m_{\widehat \theta_1; h}) \leq \frac{1}{2h} \left[\frac{11}{4}\sqrt{\frac{\log(n+2)}{n}} + \frac{\log(1/\alpha) + (13/2)}{\sqrt{n \log(n+2)}} \right] \mbox{ for all } h >0 \right\} \\
%         \geq & \mathbb{P}\Bigg\{\max\{|(P_n - P)(\left(\theta_h - h, \theta_h + h \right])|, |(P_n - P)((\widehat \theta_1 - h, \widehat \theta_1 + h ])| \}  \\
%         &  \leq \frac{1}{2} \left[\frac{11}{4}\sqrt{\frac{\log(n+2)}{n}} + \frac{\log(1/\alpha) + (13/2)}{\sqrt{n \log(n+2)}} \right] \mbox{ for all } h >0 \Bigg\} \\
%        \geq &  \mathbb{P}\left\{\sup_{I \in \mathcal{I}}  |P_n(I) - P(I)| < \frac{11}{8} \sqrt{\frac{\log(n+2)}{n}} + \frac{(\log(1/\alpha)/2) + (13/4)}{\sqrt{n \log(n+2)}}\right\} \\
%        \stackrel{\eqref{eq:simul_reiss}}{\geq} & 1 - \alpha. 
%     \end{split}
% \end{equation}
\begin{equation}
    \label{eq:sumul_reiss_2}
    \begin{split}
        &\mathbb{P}\left\{(P_n - P)(m_{ \theta_h;h} - m_{\widehat \theta_1; h}) \leq \frac{2}{2h} \sqrt{\frac{2\log(2/\alpha)}{n}}  \mbox{ for all } h >0 \right\} \\
        \geq & \mathbb{P}\Bigg\{\max\{|(P_n - P)(\left(\theta_h - h, \theta_h + h \right])|, |(P_n - P)((\widehat \theta_1 - h, \widehat \theta_1 + h ])| \}  \\
        &  \leq\sqrt{\frac{2\log(2/\alpha)}{n}}  \mbox{ for all } h >0 \Bigg\} \\
       \geq &  \mathbb{P}\left\{\sup_{I \in \mathcal{I}}  |P_n(I) - P(I)| \leq \sqrt{\frac{2\log(2/\alpha)}{n}} \right\} \\
       \stackrel{\eqref{eq:simul_DKW}}{\geq} & 1 - \alpha. 
    \end{split}
\end{equation}
In other words $\mathbb{P}(\mathcal{A}_n) \geq 1 - \alpha$ where, 
\[
\mathcal{A}_n \coloneqq \left\{ (P_n - P)(m_{ \theta_h;h} - m_{\widehat \theta_1; h}) \leq \frac{1}{h} \sqrt{\frac{2\log(2/\alpha)}{n}}  \mbox{ for all } h >0 \right\}. 
\]
We let, 
\[
\tilde{\mathrm{CS}}_{n, \alpha}^h  = \left\{\theta \in \mathbb{R} \Bigg| P_{n}(m_{\theta; h} - m_{\widehat \theta_1;h}) \leq \frac{1}{h} \sqrt{\frac{2\log(2/\alpha)}{n}}  \right\}. 
\]
We have the following bound on mis-coverage,
\begin{equation*}
    \begin{split}
       & \mathbb{P}\left( \exists h > 0 \ni \theta_h \notin \tilde{\mathrm{CS}}_{n, \alpha}^h \right) \\
        = &  \mathbb{P}\left( \exists h > 0 \ni  P_{n}(m_{\theta_h; h} - m_{\widehat \theta_1;h}) > \frac{1}{h} \sqrt{\frac{2\log(2/\alpha)}{n}}   \right) \\
        \leq & \mathbb{P}\left(\left\{ \exists h > 0 \ni  P_{n}(m_{\theta_h; h} - m_{\widehat \theta_1;h}) > \frac{1}{h} \sqrt{\frac{2\log(2/\alpha)}{n}}  \right\} \bigcap \mathcal{A}_n \right) + \mathbb{P}(\mathcal{A}_n^c) \\
        \leq & \mathbb{P} \left( \exists h > 0 \ni P(m_{\theta_h; h} - m_{\widehat \theta_1;h}) > 0\right) + \mathbb{P}(\mathcal{A}_n^c) \\
        \stackrel{(i)}{=}& 0 + \mathbb{P}(\mathcal{A}_n^c)  \\
        \stackrel{\eqref{eq:sumul_reiss_2}}{\leq}&  \alpha.
    \end{split}
\end{equation*}
% \begin{equation*}
%     \begin{split}
%       & \mathbb{P}\left(\theta_h \in \tilde{\mathrm{CS}}_{n, \alpha}^h \mbox{ for all } h > 0  \right)  \\
%     = & \mathbb{P} \left\{ P_{n}(m_{\widehat \theta_1; h} - m_{\theta_h;h}) \leq \frac{1}{2h} \left[\frac{11}{4}\sqrt{\frac{\log(n+2)}{n}} + \frac{\log(1/\alpha) + (13/2)}{\sqrt{n \log(n+2)}} \right] \mbox{ for all } h > 0\right\} \\
%     \stackrel{(i)}{=} &  \mathbb{P} \left\{ \left\{P_n  (m_{\widehat \theta_1; h} - m_{\theta_h;h}) -     \frac{1}{2h} \left[\frac{11}{4}\sqrt{\frac{\log(n+2)}{n}} + \frac{\log(1/\alpha) + (13/2)}{\sqrt{n \log(n+2)}} \right] \leq 0 \right\} \cap \{P (m_{\widehat \theta_1; h} - m_{\theta_h;h}) \geq 0 \} \mbox{ }\forall h > 0             \right\} \\
%     \geq & \mathbb{P}\left\{(P_n - P)(m_{ \widehat \theta_1;h} - m_{\theta_h; h}) \leq \frac{1}{2h} \left[\frac{11}{4}\sqrt{\frac{\log(n+2)}{n}} + \frac{\log(1/\alpha) + (13/2)}{\sqrt{n \log(n+2)}} \right] \mbox{ for all } h >0 \right\} \\
%     \stackrel{\eqref{eq:sumul_reiss_2}}{\geq}& 1 - \alpha. 
%     \end{split}
% \end{equation*}
The step-$(i)$ follows from the fact that $\theta_h$ minimizes $Pm_{\theta; h}$ for all $h > 0$. The above bound on mis-coverage implies that $\mathbb{P}( \theta_h \in \tilde{\mathrm{CS}}_{n, \alpha}^h \mbox{ for all } h > 0) \geq 1 - \alpha$. Using this fact, we can bound the coverage of $\widehat{\mathrm{CS}}_{n, \alpha}$ as follows, 
\begin{equation*}
    \begin{split}
        & \mathbb{P} \left(\theta_0 \in \widehat{\mathrm{CS}}_{n, \alpha} \right) \\
        \geq & \mathbb{P} \left( \theta_0 \in \widehat{\mathrm{CS}}_{n, \alpha}^h \mbox{ for all } h > 0 \right) \\
        = & \mathbb{P} \left( \theta_0 \in  \left\{\theta \in \mathbb{R} \Bigg| P_{n}(m_{\theta; h} - m_{\widehat \theta_1;h}) \leq \frac{1}{h} \sqrt{\frac{2\log(2/\alpha)}{n}}  \right\} \pm h \mbox{ for all } h > 0  \right) \\
        \geq &  \mathbb{P} \left( \theta_h \in  \left\{\theta \in \mathbb{R} \Bigg| P_{n}(m_{\theta; h} - m_{\widehat \theta_1;h}) \leq \frac{1}{h} \sqrt{\frac{2\log(2/\alpha)}{n}}  \right\}  \mbox{ for all } h > 0  \right) \\
        = & \mathbb{P}\left(\theta_h \in \tilde{\mathrm{CS}}_{n, \alpha}^h \mbox{ for all } h > 0   \right) \\
        \geq & 1 - \alpha. 
    \end{split}
\end{equation*}
This completes the proof of the theorem. 

\section{Proof of \Cref{thm:ada_m_estimator_width}}
\label{app:ada_M_width}
The proof of this theorem is very similar to that of \Cref{thm:m_estimator_width}. For simplicity we redefine $m_{\theta; h}(X) = -\textbf{1}\{ \theta - h < X \leq \theta + h \}$ in this proof. The only change from the proof of \Cref{thm:m_estimator_width} is in the definition of $\widehat \Gamma_n( \widehat \theta_1, \theta_0)^h$ (the superscript $h$ denotes the dependence on $h$), 
\[
    \widehat \Gamma_n( \widehat \theta_1, \theta_0)^h  = \sqrt{\frac{8\log(2/\alpha)}{n}}  + | P_n ( m_{\widehat \theta_1; h} - m_{\theta_0; h})|. 
\]
It can be shown that $\mathbb{P}^1 ( \widehat \Gamma_n( \widehat \theta_1, \theta_0)^h > \tilde \gamma_{n , \epsilon}^h) \leq \epsilon/4 $ where, 
\[
    \tilde \gamma_{n , \epsilon}^h = \frac{\sqrt{8\log(2/\alpha)} + \sqrt{8 \log(8/\epsilon)}}{\sqrt{n}} + | P ( m_{\widehat \theta_1; h} - m_{\theta_0; h})|. 
\]
From the minimax optimality of $\widehat \theta_1$ we know that there exists $A_{\epsilon} > 0$ such that $|\widehat \theta_1 - \theta_0| \leq A_{\epsilon} n^{-1/(1 + 2 \beta)}$  with probability at-least $1 - (\epsilon/4)$. Using this fact and the definition of $\tilde \gamma_{n , \epsilon}^h$ it can be shown that the following event holds with probability at-least $1 - (\epsilon/ 4)$, 
\begin{equation*}
    \begin{split}
        & \tilde \gamma_{n , \epsilon}^h \\
        = & \frac{\sqrt{8\log(2/\alpha)} + \sqrt{8 \log(8/\epsilon)}}{\sqrt{n}}  + | P ( m_{\widehat \theta_1; h} - m_{\theta_0; h})| \\
        \leq & \frac{\sqrt{8\log(2/\alpha)} + \sqrt{8 \log(8/\epsilon)}}{\sqrt{n}}  \\
        &+ 2f(\theta_0)A_{\epsilon}n^{-1/(1 + 2 \beta)} \textbf{1}\{ |\widehat{\theta}_1 - \theta_0| \leq h\} +  | P ( m_{\widehat \theta_1; h} - m_{\theta_0; h})|\textbf{1}\{| \widehat \theta_1 - \theta_0| > h \} \\
        \coloneqq & \gamma_{n, \epsilon}^h.   
    \end{split}
\end{equation*}
Following exactly the same steps as the proof of \Cref{thm:m_estimator_width} we can show that with probability at-least $1 - \epsilon$, 
\begin{equation*}
    \begin{split}
         & \mbox{Width}\left( \widehat{\mathrm{CS}}^h_{n,\alpha} \right) \\
      \leq &  \left\{2 \left(\frac{4C_0}{(\beta + 1)c_0} \right)^{1/\beta} + 4 \right\}h \\
      & + \mathcal{C} \left(\frac{2}{hc_0}\right)^{1/\beta}  \max\left\{\left[2C_2\sqrt{C_1 (hc_0/2)^{-1/\beta}}\right]^{2/(2\beta - 1)} n^{-1/(2\beta - 1)} , (\sqrt{2}C_2)^{2/\beta} n^{-\beta} \right\}  \\
      & + \mathcal{C} \left(\frac{2}{hc_0}\right)^{1/\beta} \left\{\gamma_{n , \epsilon}^h \right\} ^{1/\beta} \\
      \leq & \left\{ 2\left(\frac{4C_0}{(\beta + 1)c_0} \right)^{1/\beta} + 4 \right\}h \\ 
      & + \mathcal{C}(2/c_0)^{1/\beta} \left[2C_2\sqrt{C_1 (c_0/2)^{-1/\beta}}\right]^{2/(2\beta - 1)} n^{-1/(2\beta - 1)} h^{-2/(2\beta -1)} + \mathcal{C}(2/c_0)^{1/\beta} (\sqrt{2}C_2)^{2/\beta} n^{-\beta}h^{-1/\beta} \\
      &+ \mathcal{C}\left( \frac{2}{hc_0} \right)^{1/\beta} \left\{\frac{\sqrt{8\log(2/\alpha)} + \sqrt{8 \log(8/\epsilon)}}{\sqrt{n}}   + 2f(\theta_0)A_{\epsilon}n^{-1/(1+ 2\beta)}\textbf{1}\{|\widehat \theta_1 - \theta_0| \leq h\} \right\}^{1/\beta} \\
     & + \mathcal{C}(2/c_0)^{1/\beta}h^{-1/\beta} \left\{|P(m_{\widehat \theta_1;h} - m_{\theta_0;h})|\textbf{1}\{|\widehat \theta_1 - \theta_0| > h\}\right\}^{1/\beta}.
    \end{split}
\end{equation*}
Let $\mathfrak{C} > 0$ be such that,
\[
\frac{\sqrt{8\log(2/\alpha)} + \sqrt{8 \log(8/\epsilon)}}{\sqrt{n}}  \leq (\mathfrak{C} - 2f(\theta_0))A_{\epsilon}n^{-1/(1+ 2\beta)}.
\]
We define $h_{opt} \coloneqq \sqrt{\ell(n)} n^{-1/(1+ 2\beta)}$. The following event holds with probability at-least $1 - \epsilon$, 
\begin{equation*}
    \begin{split}
    & \mbox{Width}\left( \widehat{\mathrm{CS}}^h_{n,\alpha} \right) \\
    \leq & \left\{ 2\left(\frac{4C_0}{(\beta + 1)c_0} \right)^{1/\beta} + 4 \right\} \ell(n)^{1/2} n^{-1/(1 + 2\beta)} \\
    &+ \mathcal{C}(2/c_0)^{1/\beta} \left[2C_2\sqrt{C_1 (c_0/2)^{-1/\beta}}\right]^{2/(2\beta - 1)}  \ell(n)^{-1/(2\beta -1)} n^{-1/(1 + 2\beta)} \\
    &+ \mathcal{C}(2/c_0)^{1/\beta} (\sqrt{2}C_2)^{2/\beta} \ell(n)^{-1/2\beta} n^{-\beta + (1/(\beta(1+2\beta)))}\\
    &+ \mathcal{C}(2/c_0)^{1/\beta}\ell(n)^{-1/2\beta}n^{1/(\beta(1 + 2\beta))}\mathfrak{C}^{1/\beta}A_{\epsilon}^{1/\beta} n^{-1/(\beta(1 + 2\beta))}. 
    \end{split}
\end{equation*}
Thus with probability greater than or equal to $1 - \epsilon$ we have, 
\begin{equation*}
    \begin{split}
       & n^{1/(1+ 2\beta)}\ell(n)^{-1} \mbox{Width}\left( \widehat{\mathrm{CS}}^{h_{opt}}_{n,\alpha} \right)  \\
        \leq & \left\{ 2\left(\frac{4C_0}{(\beta + 1)c_0} \right)^{1/\beta} + 4 \right\}\ell(n)^{-1/2} + \mathcal{C}(2/c_0)^{1/\beta} \left[2C_2\sqrt{C_1 (c_0/2)^{-1/\beta}}\right]^{2/(2\beta - 1)}  \ell(n)^{-1-(1/(2\beta -1))} \\
        &+ \mathcal{C}(2/c_0)^{1/\beta} (\sqrt{2}C_2)^{2/\beta} \ell(n)^{-1-(1/2\beta)} n^{-\beta + (1/(\beta(1+2\beta))) + (1/(1+ 2\beta))} \\
        &+ \mathcal{C}(2/c_0)^{1/\beta}\mathfrak{C}^{1/\beta}A_{\epsilon}^{1/\beta}\ell( n)^{-1-(1/2\beta)}. 
    \end{split}
\end{equation*}
For $\beta \geq 1$, the exponent $-\beta + (1/(\beta(1+2\beta))) + (1/(1+ 2\beta))$ stays negative. Since $\mbox{Width}( \widehat{\mathrm{CS}}_{n,\alpha} ) = \min_{h \geq 0} \mbox{Width}( \widehat{\mathrm{CS}}^h_{n,\alpha} ) \leq \mbox{Width}( \widehat{\mathrm{CS}}^{h_{opt}}_{n,\alpha} )$, we can say that the following event holds with probability at-least $ 1- \epsilon$, 
\begin{equation*}
    \begin{split}
       & n^{1/(1+ 2\beta)}\ell(n)^{-1} \mbox{Width}\left( \widehat{\mathrm{CS}}_{n,\alpha} \right)  \\
        \leq & \left\{ 2\left(\frac{4C_0}{(\beta + 1)c_0} \right)^{1/\beta} + 4 \right\}\ell(n)^{-1/2} + \mathcal{C}(2/c_0)^{1/\beta} \left[2C_2\sqrt{C_1 (c_0/2)^{-1/\beta}}\right]^{2/(2\beta - 1)}  \ell(n)^{-1-(1/(2\beta -1))} \\
        &+ \mathcal{C}(2/c_0)^{1/\beta} (\sqrt{2}C_2)^{2/\beta} \ell(n)^{-1-(1/2\beta)} n^{-\beta + (1/(\beta(1+2\beta))) + (1/(1+ 2\beta))} \\
        &+ \mathcal{C}(2/c_0)^{1/\beta}\mathfrak{C}^{1/\beta}A_{\epsilon}^{1/\beta}\ell( n)^{-1-(1/2\beta)}. 
    \end{split}
\end{equation*}
This completes the proof that $n^{1/(1+ 2\beta)}\ell(n)^{-1}\mbox{Width}( \widehat{\mathrm{CS}}_{n,\alpha} )  = O_P(1) $. 

\section{Proof of \Cref{thm:ed_method_validity}}
\label{app:proof_ed_validity}
We know from Edelman's result \eqref{eq:edelman} that for all $X_i \in S_2$,
\begin{equation}
  \mathbb{P}\left(X_i - \left(\frac{2 }{\alpha} - 1 \right)| X_i - \widehat \theta_1| \leq \theta_0 \leq X_i + \left(\frac{2}{\alpha} - 1 \right) | X_i - \widehat \theta_1|  \right) \geq 1 - \alpha.    
\end{equation}
The above result holds because $\widehat \theta_1$ is independent of any $X_i \in S_2$. The above inequality can be re-stated as, 
\begin{equation}
     \label{eq:p_value_ed}
    \mathbb{P}\left( \frac{2}{1 + \left|\frac{X_i - \theta_0}{X_i - \widehat \theta_1} \right|} \leq \alpha \right) \leq \alpha \mbox{ for all } X_i \in S_2.  
\end{equation}
This shows that for all $X_i \in S_2$, $p_i(\theta_0)$ is a valid p-value under the null hypothesis that the mode of the underlying distribution is $\theta_0$. Moreover since $\{p_i(\theta_0)\}_{X_i \in S_2}$ are independent p-values, we can use any standard method of combining independent p-values and use that to obtain a valid confidence set for $\theta_0$. In \Cref{alg:proposed-conf-int_ed} we use Fisher's method of combining independent p-values. This completes the proof of the theorem. 

\section{Proof of \Cref{thm:ed_width}}
\label{app:proof_ed_width}
The proof is a simple application of the law of large numbers. We have the following equivalence relations, 
\begin{equation*}
        \begin{split}
            \theta \in  \widehat{\mathrm{CS}}_{n, \alpha}^{\mathrm{Ed}_p}& \iff -2\sum_{i: X_i \in S_2}\log(p_i(\theta)) < \chi^2( 1- \alpha, 2n)  \\
            & \iff  \frac{1}{n}  \sum_{i: X_i \in S_2} \log\left( \frac{1 + \left|\frac{X_i - \theta}{X_i - \widehat \theta_1} \right|}{2} \right) < \frac{\chi^2( 1- \alpha, 2n) }{2n}. 
        \end{split}
    \end{equation*}
As $n \rightarrow \infty$ the term in the left hand side converges almost surely to $\mathbb{E}[\log(1 + |(X - \theta)/(X - \theta_0)|) - \log(2)]$ by strong law of large numbers (we also use the fact that $\widehat \theta_1$ is a consistent estimator of $\theta_0$). The right hand side converges to $\mathbb{E}Z^2 = 1$ ($Z$ is standard normal random variable) by strong law of large numbers. Thus as $m-n, n \rightarrow \infty$, 
\[
\widehat{\mathrm{CS}}_{n, \alpha}^{\mathrm{Ed}_p} \stackrel{a.s.}{\rightarrow} \left\{\theta \Bigg| \mathbb{E}_F \left[ \log\left(1 +  \left|\frac{X - \theta}{X - \theta_0} \right| \right)\right] < 1+ \log 2\right\}.
\]
This completes the proof of the theorem. 

\section{Proof of \Cref{thm:gamma_ed_method_validity}}
 \label{app:gamma_validity}
We know from \Cref{lem:minkow_gamma} that if $X_1, \cdots, X_n$ follows a $\gamma$-unimodal distribution about mode $\theta_0$, then $\|X_1 - \theta_0\|_2^{\gamma}, \cdots, \|X_n - \theta_0\|_2^{\gamma}$ follow unimodal distribution on $\mathbb{R}$ about mode $0$. From the definition of $\widehat{\mathrm{CS}}_{n,\alpha}^{\theta_0; \gamma} $ we get, 
\[
\mathbb{P}\left(0 \in \widehat{\mathrm{CS}}_{n,\alpha}^{\theta_0; \gamma} \right) \geq 1- \alpha. 
\]
% Using \Cref{thm:ed_method_validity} on the transformed data $\{ \|X_1 - \theta_0\|_2^{\gamma}, \cdots, \|X_n - \theta_0\|_2^{\gamma}\}$ we get, 
% \begin{equation}
% \label{eq:gamma_long}
%     \begin{split}
% \mathbb{P}&\left( 0 \in \left[\|\widehat X_{(1)}^{\theta_0} - \theta_0\|_2^{\gamma} - \left(\frac{2 \ell(n)}{\alpha} - 1 \right) |\|\widehat X_{(1)}^{\theta_0} - \theta_0\|_2^{\gamma} - \|\widehat \theta_1 - \theta_0\|_2^{\gamma}|, \right. \right. \\
%    & \left. \left. \|\widehat X_{(1)}^{\theta_0} - \theta_0\|_2^{\gamma} + \left(\frac{2 \ell(n)}{\alpha} - 1 \right) |\|\widehat X_{(1)}^{\theta_0} - \theta_0\|_2^{\gamma} - \|\widehat \theta_1 - \theta_0\|_2^{\gamma}| \right] \right) \geq 1 - \alpha.        
%     \end{split}
% \end{equation}
The above coverage guarantee holds because of the validity of the algorithm $\mathcal{A}(\cdot, \cdot)$. Therefore we can bound the coverage probability as follows, 
\begin{equation*}
    \begin{split}
        &\mathbb{P}\left(\theta_0 \in \widehat{\mathrm{CS}}_{n,\alpha}^{ \gamma}  \right) \\
     = & \mathbb{P}\left(\theta_0 \in \left\{\theta \in \mathbb{R}^d\Bigg|  0 \in \widehat{\mathrm{CS}}_{n,\alpha}^{\theta; \gamma}   \right\}  \right) \\
     = & \mathbb{P}\left(0 \in \widehat{\mathrm{CS}}_{n,\alpha}^{\theta_0; \gamma} \right) \\
     \geq& 1- \alpha. 
    \end{split}
\end{equation*}
This completes the proof of the theorem. 
\end{document}